\documentclass[pdftex,12pt,a4paper]{article}

\usepackage[utf8]{inputenc}
\usepackage{amsmath}
\usepackage{amssymb}
\usepackage{amsthm}
\usepackage{tabularx}
\usepackage{xcolor}
\usepackage{tikz}
\usepackage{mathtools}

\newtheorem{theorem}{Theorem}[section]
\newtheorem{lemma}[theorem]{Lemma}
\newtheorem{proposition}[theorem]{Proposition}
\newtheorem{corollary}[theorem]{Corollary}
\newtheorem{remark}[theorem]{Remark}
\newtheorem{definition}[theorem]{Definition}

\newcommand{\Sz}{\mathrm{Sz}}
\newcommand{\Ree}{\mathrm{Ree}}
\newcommand{\PSL}{\mathrm{PSL}}
\newcommand{\PGU}{\mathrm{PGU}}
\newcommand{\id}{\mathrm{id}}
\newcommand{\ord}{\mathrm{ord}}
\newcommand{\Aut}{\mathrm{Aut}}
\newcommand{\lcm}{\mathrm{lcm}}
\newcommand{\im}{\mathrm{im}}

\newcommand{\Z}{\mathbb{Z}}

\title{Classification of all Galois subcovers of the Skabelund maximal curves}
\date{}

\author{Peter Beelen, Leonardo Landi and Maria Montanucci}

\begin{document}
\maketitle

\begin{abstract}
In 2017 Skabelund constructed two new examples of maximal curves $\tilde{\mathcal{S}}_q$ and $\tilde{\mathcal{R}}_q$ as covers of the Suzuki and Ree curves, respectively. The resulting Skabelund curves are analogous to  the Giulietti-Korchmáros cover of the Hermitian curve.
In this paper a complete characterization of all Galois subcovers of the Skabelund curves $\tilde{\mathcal{S}}_q$ and $\tilde{\mathcal{R}}_q$ is given.
Calculating the genera of the corresponding curves, we find new additions to the list of known genera of maximal curves over finite fields.
\end{abstract}

\noindent
AMS: 11G20, 14H25, 14H37

\vspace{1ex}

\section{Introduction}

Let $\mathbb F_{Q^2}$ be a finite field with $Q^2$ elements where $Q$ is a power of a prime $p$.
For an algebraic (projective, absolutely irreducible, nonsingular) curve $\mathcal C$ of genus $g(\mathcal C)$ over $\mathbb F_{Q^2}$, the Hasse--Weil bound states that:
$$N(\mathcal C)\leq Q^2+1+2g(\mathcal C)Q,$$
where $N(\mathcal C)$ denotes the number of rational points of $\mathcal C$. The curve $\mathcal C$ is called maximal if $N(\mathcal C)= Q^2+1+2g(\mathcal C)Q,$ that is, if $\mathcal C$ has the largest possible number of rational points that it can have according to
the value $g(\mathcal C)$ of its genus.
Maximal curves have interesting properties and have also been intensively investigated during the last years for their applications in coding theory.
Surveys on maximal curves are found in \cite{11, 12, 16, 34, 35} and \cite[Chapter 10]{HKT}; see also \cite{9, 10, Sti}.

Important examples of maximal curves are the so-called Deligne--Lusztig curves, namely the Hermitian, Suzuki and Ree curves.
All these curves have a large automorphism group when compared with their genus, as they do not satisfy the classical Hurwitz bound $|\Aut(\mathcal C)| \leq 84(g(\mathcal C)-1)$.
Also their automorphism groups are well-studied examples of finite $2$-transitive groups, namely the projective unitary group $\PGU(3,q)$, the Suzuki group $\Sz(q)$ and the Ree group $\Ree(q)$, respectively.
The Deligne--Lusztig curves have been intensively investigated in the last decades.
Among other reasons, such as their connection with class field theory (see \cite{21}) and their applications to coding theory, the interest on this class of maximal curves is motivated by the fact that a subcover of a maximal curve over the same field of definition is maximal by a theorem of Serre \cite{Lach}. This implies that when given a maximal curve $\mathcal C$ over $\mathbb{F}_{Q^2}$ with many automorphisms,
computing Galois subcovers corresponding to subgroup $H$ of $\Aut(\mathcal C)$ gives rise to many examples of maximal curves.

Since all maximal subgroups of $\PGU(3,q)$, $\Sz(q)$ and $\Ree(q)$ are known, subgroups of these and the corresponding Galois subcovers have been studied in various papers, see for example \cite{AQ,BMXY,DMZ,GSX,MZ, MZq14}.
Many genera of maximal curves have been obtained in this way, adding to the understanding of the genus spectrum of maximal curves.

In \cite{GK} Giulietti and Korchm\'aros introduced a new maximal curve (known as the GK curve) over finite fields $\mathbb F_{q^6}$, which are not subcover of the Hermitian curve over the
corresponding field for $q>2$. Surprisingly, the GK curve was constructed as a Galois cover of the Hermitian function field over $\mathbb{F}_{q^2}$. Considering subcovers of the GK curve, gives rise
to new genera of maximal curves. Such examples were found initially in \cite{FG2,GQZ}. Later, the GK curve was generalized in \cite{GGS} to a family of maximal curves over finite fields
$\mathbb F_{q^{2n}}$ with $n$ odd. These maximal curves are often called the Garcia--G\"{u}neri--Stichtenoth (GGS) curves. All subgroups of the automorphism groups of these curves were
classified in \cite{ABB}, but before that several Galois subcovers were already determined in \cite{GOS}.

Recently a second generalization of the GK curve was discovered \cite{BM}. As for the GGS curves, for each odd $n$ a maximal curve $K_n$ was found over $\mathbb F_{q^{2n}}.$
Though the genus of $K_n$ is equal to that of the corresponding GGS curve, their automorphism groups are different. A preliminary study in \cite{BM} already revealed that new genera of maximal curves can be
obtained by considering Galois subcovers of $K_n.$
A more detailed study of Galois subcovers of the second generalization of the GK function field was provided in \cite{BM2}.

In \cite{Sk}, Skabelund constructed two Galois covers of the Suzuki and Ree curves, $\tilde{\mathcal{S}}_q$ and $\tilde{\mathcal{R}}_q$, reproducing the way in which the GK curve was constructed as a
cover of the Hermitian curve on the two other Deligne--Lusztig curves. The curves $\tilde{\mathcal{S}}_q$ and $\tilde{\mathcal{R}}_q$ can be described as follows.

For any $s \in \mathbb{N}$, $s \geq 1$, let $q_0 = 2^s$, $q = 2q_0^2 = 2^{2s+1}$ and $m = q - 2q_0 + 1$.
The Skabelund curve $\tilde{\mathcal{S}}_q$ is given by $\tilde{\mathcal{S}}_q=\mathbb{F}_{q^4}(x,y,t)$, where
$$
\begin{cases}
  y^q + y = x^{q_0} (x^q + x) \\
  t^m = x^q + x.
\end{cases} $$
The curve $\tilde{\mathcal{S}}_q$ is maximal over the field $\mathbb{F}_{q^4}$. Its automorphism group is $\Aut(\tilde{\mathcal{S}}_q) = \Sz(q) \times C_m$, see \cite{GMQZ}.

Similarly, for any $s \in \mathbb{N}$, $s \geq 1$, let $q_0 = 3^s$, $q = 3q_0^2 = 3^{2s+1}$ and $m = q - 3q_0 + 1$.
The Skabelund curve $\tilde{\mathcal{R}}_q$ is given by $\tilde{\mathcal{R}}_q=\mathbb{F}_{q^6}(x,y,z,t)$, where

$$
\begin{cases}
  y^q - y = x^{q_0} (x^q - x) \\
  z^q - z = x^{2q_0} (x^q - x) \\
  t^m = x^q - x.
\end{cases} $$
$\tilde{\mathcal{R}}_q$ is maximal over  $\mathbb{F}_{q^6}$. Its automorphism group is $\Aut(\tilde{\mathcal{R}}_q) = \Ree(q) \times C_m$, see \cite{GMQZ}.
A partial description of Galois subcovers of $\tilde{\mathcal{S}}_q$ and $\tilde{\mathcal{R}}_q$ was given in \cite{GMQZ} where only subgroups of $Aut(\tilde{\mathcal{S}}_q)$ (resp. $Aut(\tilde{\mathcal{R}}_q)$), of type
$K \times H$ with $K \leq Sz(q)$ (resp. $K \leq \Ree(q)$) and $H \leq C_m$ were considered.

In this paper, the complete classification of Galois subcovers of $\tilde{\mathcal{S}}_q$ and $\tilde{\mathcal{R}}_q$ is given.
The corresponding genera are computed for all values of $q$, giving new genera of maximal curves for specific values of $q$ (see Tables \ref{tab:cm_times_cm_Sz}, \ref{tab:cm_times_cm_Ree}, and \ref{tab:Ree3timesCmskew}).
To the best of our knowledge the genera given in these tables are new.
We have checked that these values are not contained in and cannot be obtained using results from \cite{AQ,ABB,BMXY,BM,BM2,CO,CO2,CKT,CKT2,DMZ,DO,FG,FG2,GSX,GV,GV2,GHKT,GKT,GMQZ,GOS,MX,MZ,MZ2,MZq14}.
The paper is organized as follows: in section two, we classify Galois subcovers of $\tilde{\mathcal{S}}_q$, while in section three, we achieve this for $\tilde{\mathcal{R}}_q$.

\section{Galois subcovers of $\tilde{\mathcal{S}}_q$}

In this section, we complete the study of subcovers, initiated in \cite{GMQZ}, of $\tilde{\mathcal{S}}_q$ of the form $\tilde{\mathcal{S}}_q/H$, where $H$ is a subgroup of $\Aut(\tilde{\mathcal{S}}_q)$, computing the genus
of $\tilde{\mathcal{S}}_q/H$.
Throughout the section $s \geq 1$ is a fixed integer, $q_0 := 2^s$, $q := 2q_0^2$ and $m := q - 2q_0 + 1$ unless explicitly stated otherwise.
 It is well known that $\Aut(\tilde{\mathcal{S}}_q)=\Sz(q)$ and it was shown in \cite{GMQZ} that $\Aut(\tilde{\mathcal{S}}_q)=\Sz(q) \times C_m$. Here $\Sz(q)$ is the Suzuki group over $\mathbb{F}_q$ and $C_m =\langle \tau \rangle$, with $\tau(x)=x$, $\tau(y)=y$, and $\tau(t)=\lambda t$, where $\lambda \in \mathbb{F}_{q^4}$
is an element of multiplicative order $m$. We start by proving the following proposition, which is a refinement of Lemma 3.3 from \cite{Sk}.
\begin{proposition}\label{prop:rationalliftSuzuki}
Every automorphism of $\mathcal{S}_q$ can be lifted to an automorphism of $\tilde{\mathcal S}_q$ defined over $\mathbb{F}_q$ in a unique way. The resulting collection of automorphisms forms a group isomorphic to $\Sz(q).$
\end{proposition}
\begin{proof}
The proof from \cite{Sk} mentions that automorphisms $\psi_{abc}\in \Aut(\mathcal S_q)$ defined by $\psi_{abc}(x)=ax+b$ and $\psi_{abc}(y)=a^{q_0+1}y+b^{q_0}x+c$, with $a\in\mathbb{F}_q^*$ and $b,c \in \mathbb{F}_q$, can be lifted to
automorphisms $\psi$ of $\tilde{\mathcal S}_q$ by defining $\psi(t)=\alpha t$, where $\alpha^m=a$ for a suitably chosen $\alpha \in \mathbb{F}_{q^4}$, and that the automorphism $\phi \in \Aut(\mathcal S_q)$ defined by
$\phi(x)=z/w$ and $\phi(y)=y/w$ can be lifted to an automorphism of $\tilde{\mathcal S}_q$ by defining $\phi(t)=t/w$.
Here $z:=y^{2q_0}+x^{2q_0+1}$ and $w:=xy^{2q_0}+z^{2q_0}$.
Since the $\psi_{abc}$ and $\phi$ generate $\Aut(\mathcal S_q)$, it is then concluded in \cite{Sk} that any automorphism in $\Aut(\mathcal S_q)$ can be lifted to one of
$\Aut(\mathcal S_q)$ when the field of definition is extended to $\mathbb{F}_{q^4}$.
However, it is clear that the lift of $\phi$ is actually defined over $\mathbb{F}_q$. Moreover, since $\gcd(m,q-1)=1$, the $m$-th power map
acts as a permutation on $\mathbb{F}_q^*$, implying that for each $a \in \mathbb{F}_q^*$ the equation $\alpha^m=a$ has a unique solution in $\alpha \in \mathbb{F}_q^*$.
Choosing this $\alpha$ to lift $\psi_{abc}$, we obtain a lift of
$\psi_{abc}$ defined over $\mathbb{F}_q$.
Using as in \cite{Sk} that the $\psi_{abc}$ and $\phi$ generate $\Aut(\mathcal S_q)$, it follows that any automorphism $\sigma \in \mathcal{S}_q$ can be lifted to an automorphism $\tilde{\sigma} \in \tilde{\mathcal S}_q$
defined over $\mathbb{F}_q$. Moreover, it follows from the above procedure that $\tilde{\sigma}(t)=f(x,y) t$ for some function $f(x,y)$ on $\mathcal S_q$ defined over $\mathbb{F}_q.$
All lifts of $\sigma$ are of the form $\tilde{\sigma} \tau^{k}$ for $k=0,\dots,m-1$, then $\tilde{\sigma} \tau^{k}=f(x,y) \lambda^k t$, which is defined over $\mathbb{F}_q$ is and only if $\lambda^k=\id_{C_m}.$
This proves the first part of the proposition as well as the fact that $\Aut(\tilde{\mathcal{S}}_q)$ contains exactly $|\Sz(q)|$ many elements defined over $\mathbb{F}_q$.
The natural map from these elements to $\Aut(\mathcal S_q)$ ``forgetting'' the action on $t$, is a bijective group homomorphism, whence the second part of the proposition follows.
\end{proof}

We will call the lift of $\sigma \in \Aut(\mathcal{S}_q)$ described in Proposition \ref{prop:rationalliftSuzuki} the $\mathbb{F}_q$-rational lift of $\sigma$.
With slight abuse of notation, we denote this lift again by $\sigma$ and think of $\Sz(q)$ as a subset of $\Aut(\tilde{\mathcal{S}}_q)$.
The fact already proved in \cite{GMQZ} that $\Aut(\tilde{\mathcal{S}}_q)=\Sz(q) \times C_m$ now also follows quite easily: indeed have constructed a natural copy of $\Sz(q)$ inside $\Aut(\tilde{\mathcal{S}}_q)$,
$\tau$ commutes with any element in $\Sz(q)$, and any element in $\Aut(\tilde{\mathcal{S}}_q)$ is of the form $\sigma \tau^k$ for $\sigma \in \Sz(q)$.

Any non-trivial subgroup of $\Aut(\tilde{\mathcal{S}}_q)$ is contained in one of its maximal subgroups, it is sufficient to consider subgroups of maximal subgroups.
Since $\Aut(\tilde{\mathcal{S}}_q)=\Sz(q) \times C_m$, the following simple lemma will be very convenient.

\begin{lemma}\label{lem:max}
Let $H \subseteq G$ be a subgroup of a direct product of groups $G=G_1 \times G_2$ and for $i=1,2$, define $\pi_i: H \to G_i$ as $\pi_i(g_1,g_2)=g_i$.
If $H$ is a maximal subgroup of $G$, then either $H=H_1 \times G_2$, with $H_1$ a maximal subgroup of $G_1$, or $H=G_1 \times H_2$, with $H_2$ a maximal subgroup of $G_2$, or $|H|$ is a multiple of $\lcm(|G_1|,|G_2|)$
and for $i=1,2$, $|\ker(\pi_i)|$ is a multiple of $|G_{3-i}|/\gcd(|G_1|,|G_2|).$
\end{lemma}
\begin{proof}
Let $H \subset G_1 \times G_2$ be a maximal subgroup and for $i=1,2$ write $H_i=\im(\pi_i).$ If is clear that $H \subseteq H_1 \times H_2.$
Since $H$ is maximal, either $H=H_1 \times H_2$ or $H_1 \times H_2=G_1 \times G_2$. In the former case, we deduce from the maximality of $H$ that either $H_1=G_1$ and $H_2$ is a maximal subgroup of $G_2$,
or $H_1$ is a maximal subgroup of $G_1$ and $H_2=G_2.$ In the latter case, we see that $G_i=H_i\cong H/\ker(\pi_i),$ where we used the isomorphism theorem.
This implies that $H$ is both a multiple of $|G_1|$ and of $|G_2|$ and hence of the least common multiple of the two. Since more specifically $|H|=|\ker(\pi_i)|\cdot |G_i|,$ we also see that $|\ker(\pi_i)|$
is a multiple of $\lcm(|G_1|,|G_2|)/|G_i|=|G_{3-i}|/\gcd(|G_1|,|G_2|)$ for $i=1,2$.
\end{proof}

It is trivial that the maximal subgroups of $C_m=\langle \tau \rangle$ are all of the form $\langle \tau^p \rangle$, where $p$ is a prime dividing $m$. The maximal subgroups of the Suzuki group are well known and classified
up to conjugation in the following theorem. See \cite{L} Theorem 4.12 or \cite{HHP}, Theorem 3.1 for details.

\begin{theorem}
  \label{theorem:classification_subgroups_Sz}
  Up to conjugation, the Suzuki group $\Sz(q)$ has the following maximal subgroups.
  \begin{enumerate}
  \item The Frobenius group $F$ of order $q^2(q-1)$.
  \item The dihedral group $B_0$ of order $2(q-1)$.
  \item The normalizer $N_-$ of a cyclic Singer group $\Sigma_-$ with $|\Sigma_-| = q-2q_0+1$ and $|N_-| = 4 \cdot |\Sigma_-|$.
  \item The normalizer $N_+$ of a cyclic Singer group $\Sigma_+$ with $|\Sigma_+| = q+2q_0+1$ and $|N_+| = 4 \cdot |\Sigma_+|$.
  \item The Suzuki groups $\Sz(\hat{q})$ for $q = \hat{q}^h$, with $1< h < 2s+1$ and $h$ a prime.
  \end{enumerate}
Further, any subgroup of $\Sz(q)$ is either isomorphic to $\Sz(\hat{q})$ for $q=\hat{q}^k$ and $1 \le k < 2s+1$ or conjugated to a subgroup of one of $F, B_0, N_-, or N_+$.
\end{theorem}

Lemma \ref{lem:max} and Theorem \ref{theorem:classification_subgroups_Sz} allow us to describe all maximal subgroups of $\Aut(\tilde{\mathcal{S}}_q)=\Sz(q)\times C_m$. Actually, we obtain the following slightly stronger result on subgroups of $\Sz(q) \times C_m$:
\begin{corollary}\label{cor:subgroupsautskabelund}
Any subgroup $H \subset \Sz(q)\times C_m$ is either of the form $\Sz(q) \times C_n$, with $n|m$ and $C_n \subseteq C_m$ the unique subgroup of order $n$, or contained in $M \times C_m$ with $M$ a maximal subgroup of $\Sz(q).$
\end{corollary}
\begin{proof}
Let $H$ be a subgroup of $\Sz(q)\times C_m$. According to Lemma \ref{lem:max}, we can conclude that one of the following three cases will hold: either it is contained in $M \times C_m$ with $M$ a maximal subgroup of $\Sz(q)$, or contained in $\Sz(q) \times C_{m/p}$ with $p$ a prime dividing $m$ and $C_{m/p}$ the unique subgroup of $C_m$ of order $m/p$, or contained in a maximal subgroup $K$ for which $|\ker(\pi_2)|$ is a multiple of $|\Sz(q)|/m$.

In the first case, there is nothing left to prove. Now consider the third case. Since $\ker(\pi_2)=K \cap (\Sz(q) \times \{\id_{C_m}\})$, it can be identified with a subgroup of $\Sz(q).$ However, Theorem \ref{theorem:classification_subgroups_Sz} implies that the only subgroup of $\Sz(q)$ that has cardinality a multiple of $|\Sz(q)|/m=(q+2q_0+1)q^2(q-1)$ is $\Sz(q)$ itself. Hence $\ker(\pi_2)=\Sz(q)\times \{\id_{C_m}\} \subseteq K,$ which implies that $K=\Sz(q) \times C_{m/p}$ for some $p$ dividing $m$. Hence we arrive at the same conclusion as in the second case and should consider the case that $H$ is a subgroup of $\Sz(q) \times C_{m/p}.$ If $H=\Sz(q) \times C_{m/p}$, we are done. Otherwise a similar application of Lemma \ref{lem:max} and Theorem \ref{theorem:classification_subgroups_Sz} shows that in this case either $H$ is contained in $M \times C_{m/p}$ for some maximal subgroup $M$ of $\Sz(q)$, or that $H$ is a subgroup of $\Sz(q) \times C_{m/(pp')}$ for some prime $p'$ dividing $m/p'$. In the first case or in case that $H=\Sz(q) \times C_{m/(pp')},$ we are done, otherwise we continue dividing prime factors of $m$ out till we arrive at the case that $H$ is contained in $\Sz(q) \times \{\id_{C_m}\}$. At this point, we see that either $H=\Sz(q)\times \{\id_{C_m}\}$, or contained in $M \times \{\id_{C_m}\}$ for some maximal subgroup $M$ of $\Sz(q)$. This proves the corollary.
\end{proof}

In \cite{GMQZ}, the genus of the quotient curve $\tilde{\mathcal{S}}_q/H$ is computed when $H$ is one of following subgroups of $\Aut(\tilde{\mathcal{S}}_q)$:
\begin{itemize}
\item $F \times C_m$ or one of its subgroups;
\item $N_+ \times C_m$ or one of its subgroups;
\item $N_- \times C_m$ or one of its subgroups of the form $K \times C_n$ with $K$ subgroup of $N_-$ and $n$ dividing $m$;
\item $\Sz(\hat{q}) \times C_n$ for suitable $\hat{q}$ and for $n$ dividing $m$.
\end{itemize}

Corollary \ref{cor:subgroupsautskabelund} implies that the only cases where the genus of $\tilde{\mathcal{S}}_q/H$ has not been computed yet are if $H$ is one of the missing subgroups of $N_- \times C_m$ or a subgroup of
$B_0 \times C_m$. To complete these cases, we use the same approach as in \cite{GMQZ}. For a subgroup $H$ of $\Aut(\tilde{\mathcal{S}}_q)$, let $g_H$ be the genus of the quotient curve $\tilde{\mathcal{S}}_q/H$.
By the Riemann--Hurwitz formula (see \cite{Sti}, Theorem 3.4.13) applied to the cover $\tilde{\mathcal{S}}_q \to \tilde{\mathcal{S}}_q/H$, we have
\begin{equation*}\label{eq:Riemann--Hurwitz}
(q^2+1)(q-2) = |H| \cdot (2g_H-2) + \Delta_H,
\end{equation*}
where $|H|$ is the order of $H$ and $\Delta_H$ is the degree of the different divisor. By the Hilbert's different formula, $\Delta_H$ can be expressed as
$$ \Delta_H = \sum_{\substack{\omega \in H \\ \omega \neq \id}} \iota(\omega), $$
where
$$ \iota(\omega) = \sum_{\substack{P \in \tilde{\mathcal{S}}_q \\ \omega(P) = P}} | \{ i \in \Z_{\ge 0} \,:\, \omega \in H_P^{(i)} \} | $$
and $H_P^{(i)}$ is the $i$-th ramification group of the Galois cover $\tilde{\mathcal{S}}_q \to \tilde{\mathcal{S}}_q/H$ at $P$. See \cite{Sti}, Definition 3.8.4 and Theorem 3.8.7 for details. For each $\omega \in \Aut(\tilde{\mathcal{S}}_q)$, the quantity $\iota(\omega)$ is in principle computed in \cite{GMQZ}, Theorem 26; however, this theorem contains a mistake and for the convenience of the reader we give the correct formulation as well as a proof of the corrected part. For a group element $g \in G$, we denote by $\ord(g)$ the order of $g$.
%
\begin{theorem}
  \label{theorem_26}
Let $\{ \id_{\Sz(q)} \} \times C_m = \langle \tau \rangle$. Then $\iota(\tau^k) = q^2 + 1$ for all $k = 1, \dots, m-1$.
Further, let $\sigma \in \Sz(q) \times \{ \id_{C_m} \}$, $\sigma \neq \id$. Then exactly one of the following cases occurs:

\bigskip
\noindent
1. $\ord(\sigma) = 2$, $\iota(\sigma) = m(2q_0+1)+1$, and $\iota(\sigma \tau^k) = 1$ for all $k = 1, \dots, m-1$;

\bigskip
\noindent
2. $\ord(\sigma) = 4$, $\iota(\sigma) = m+1$, and $\iota(\sigma \tau^k) = 1$ for all $k = 1, \dots, m-1$;

\bigskip
\noindent
3. $\ord(\sigma) \mid (q-1)$, $\iota(\sigma \tau^k) = 2$ for all $k = 0, \dots, m-1$;

\bigskip
\noindent
4.  $\ord(\sigma) \mid (q + 2q_0 + 1)$, $\iota(\sigma \tau^k) = 0$ for all $k = 0, \dots, m-1$;

\bigskip
\noindent
5. $\ord(\sigma) \mid (q - 2q_0 + 1)$, $\iota(\sigma)=0$, $\iota(\sigma \tau^j) = m$ for exactly four distinct $j \in \{ 1, \dots, m-1 \}$, and $\iota(\sigma \tau^j) = 0$ for all other $j$ between $1$ and $m-1$.
%
\end{theorem}
\begin{proof}
Only the statements about $\iota(\sigma \tau^j)$ for $j=1,\dots,m-1$ in the fifth item need a proof, the rest of the theorem being identical to \cite{GMQZ}, Theorem 26.

Let $\sigma \in \Sigma_- \setminus \{\id\}$. Then $\sigma$ fixes an $\mathbb{F}_{q^4}$-rational, not $\mathbb{F}_q$-rational, point $P$ of the Suzuki curve $\mathcal{S}_q$.
This point $P$ will have certain affine coordinates $(x(P),y(P))=(a,b)$ and $\sigma$ also fixes all the $q$-Frobenius conjugates of $P$, $(a^{q^j},b^{q^j})$ where $j=1,2,3$.
We have seen that $\Sz(q)$ can be lifted to a subgroup of $\Aut(\tilde{\mathcal S}_q)$. We denote the corresponding lift of $\sigma \in \Sz(q)$ again by $\sigma$ for convenience.

First we calculate the possibilities for $\sigma(t)$, where $t^m=x^q+x$.
The orbit of $m$ points above $(a^{q^j},b^{q^j})$ where $j=0,1,2,3$ in the cover $\tilde{\mathcal{S}}_q \to \mathcal{S}_q$ is
$$O_j:=\{(a^{q^j},b^{q^j},(\lambda^i c)^{q^j}) \mid \lambda^m=1, \ i=0,\ldots,m-1, \ c^m=a^q+a\}.$$
Note that
$$(x^q+x)_{\mathcal{S}_q}=\sum_{\alpha^q+\alpha=\beta^q+\beta=0} P_{(\alpha,\beta)} - q^2 P_\infty.$$
Since $o:=\{P_{(\alpha,\beta)} \mid \alpha^q+\alpha=\beta^q+\beta=0 \} \cup \{P_\infty\}$ is an orbit of $Aut(\mathcal{S}_q)$, the result is
$$(\sigma(x^q+x))_{\mathcal{S}_q}=\sum_{\substack{\alpha^q+\alpha=\beta^q+\beta=0 \\ \alpha \ne \alpha_1, \ \beta \ne \beta_1}} P_{(\alpha,\beta)}+P_\infty - q^2 P_{(\alpha_1,\beta_1)},$$
with $P_{(\alpha_1,\beta_1)}=\sigma(P_\infty).$
Hence
$$\left( \frac{\sigma(x^q+x)}{x^q+x}\right)_{\mathcal{S}_q}=(q^2+1)(P_\infty-P_{(\alpha_1,\beta_1)})=\big( \tilde{w}^{-m}\big)_{\mathcal{S}_q},$$
where
$$\tilde{w}=\alpha_1 (\alpha_1^{2q_0}x+z+\beta_1^{2q_0})+\beta_1^{2q_0}x+w+\beta_1^2+\alpha_1^{2q_0+2},$$
with $z:=y^{2q_0}+x^{2q_0+1}$ and $w:=xy^{2q_0}+z^{2q_0}$,
see equation (6) in \cite{BLM}.
We may conclude that there exists a constant $\delta \in \mathbb{F}^*_{q}$ such that
$$\sigma(t)^m=\sigma(x^q+x)=\delta \frac{x^q+x}{\tilde{w}^m}=\delta \bigg( \frac{t}{\tilde{w}}\bigg)^m,$$
so that $\sigma(t)=\gamma t/\tilde{w},$ for some $\gamma \in \mathbb{F}^*_{q^4}$ such that $\gamma^m=\delta$.

Note that this implies for all $k=0,\ldots,m-1$,
$$\sigma\tau^k(x)=\sigma(x), \ \sigma\tau^k(y)=\sigma(y), \ \sigma\tau^k(t)=\gamma \lambda^k \frac{t}{\tilde{w}},$$
where $\lambda \in \mathbb{F}^*_{q^4}$ is an element of multiplicative order $m$.

Now let $\tilde{P} \in O_1$ be a point lying above $P$ in the cover $\tilde{\mathcal{S}}_q \to \mathcal{S}_q$ and write $(a,b,c):=(x(\tilde{P}),y(\tilde{P}),t(\tilde{P})).$
Suppose that $k$ is chosen such that $\sigma \circ \tau^k$ fixes $\tilde{P}$. Since $C_m$ acts on $O_1$ faithfully in a cyclic way, such $k$ exists and is unique.
Then $c=\gamma \lambda^k c/\tilde{w}(a,b)$ and hence
$$\gamma \lambda^k=\tilde{w}(a,b).$$

Clearly this implies that $\sigma\tau^k$ fixes the orbit $O_0$ point-wise, as the $t$-coordinate of the points in $O_0$ is of type $\lambda^j c$ for $j=0,\ldots,m-1$.
We want to show that no point in $O_i$ is fixed by $\sigma\tau^k$ for $i=1,2,3$.
This is equivalent to showing that $\gamma\lambda^k c^{q^i}/\tilde{w}(a^{q^i},b^{q^i}) \ne c^{q^i}$ as again the $z$-coordinate of all the other points in $O_i$ is a constant multiple of $c^{q^i}$.
Since $\tilde{w}(a^q,b^q)=\tilde{w}(a,b)^q$, the obtained quantity is equal to $c^{q^i}$ if and only if
$$\tilde{w}(a,b)=\gamma \lambda^k=\tilde{w}(a,b)^{q^i},$$
that is, if and only if $\tilde{w}(a,b) \in \mathbb{F}_{q^i}$.

Now note that $\tilde{w}(a,b)^m=(\gamma \lambda^k)^m=\delta$ for some $\delta \in \mathbb{F}_q^*.$ Since for $i=1,2,3$, $\gcd(q^i-1,m)=1$, we see for $i=1,2,3$ that $\tilde{w}(a,b) \in \mathbb{F}_{q^i}$ implies that $\tilde{w}(a,b) \in \mathbb{F}_{q}$.
However, as we will see in a moment, the function $\tilde{w}^q+\tilde{w}$ has only $\mathbb{F}_q$-rational zeros, so this cannot occur. Indeed, a direct computation shows that
\[{\tilde{w}^q+\tilde{w}}=(x^q+x)[\alpha_1^{q_0}(x+\alpha_1)+\beta_1+y]^{2q_0}.\]
Hence any zero of ${\tilde{w}^q+\tilde{w}}$ is a zero of $x^q+x$, which are $\mathbb{F}_q$-rational, or a zero of $\tilde{y}:=\alpha_1^{q_0}(x+\alpha_1)+\beta_1+y$, which describes the tangent line of $\mathcal{S}_q$ at $P_{\alpha_1,\beta_1}.$
Since $\tilde{y}^q+\tilde{y}=(x+\alpha_1)^{q_0}(x^q+x),$ any zero of $\tilde{y}^q+\tilde{y}$, and hence of $\tilde{y}$, is $\mathbb{F}_q$-rational as well.

We may conclude that $\iota(\sigma \tau^k)=m$. Starting with a point in one of the other orbits $O_1$, $O_2$, $O_3$, one can similarly find a unique $k$, a different one for each orbit, such that $\iota(\sigma \tau^k)=m$.
\end{proof}

We will now supplement this theorem with a result, which is very convenient from a computational perspective. More precisely, in the fifth case of Theorem \ref{theorem_26}, we determine the four special values of $j$
mentioned there in case $\sigma$ is the $\mathbb{F}_q$-rational lift of an automorphism of $\mathcal S_q$.

\begin{proposition}\label{prop:fourjs}
Let $\sigma$ be the $\mathbb{F}_q$-rational lift of an element of $\Aut(\mathcal S_q)$ of order $q - 2q_0 + 1$.
Then there exists a choice of the generator $\tau$ of $C_m$ such that the four values of $j$ for which $\iota(\sigma \tau^j) = m$ are $q^d \bmod{m}$ for $d = 0,1,2,3$.
\end{proposition}
\begin{proof}
From the proof of Theorem \ref{theorem_26}, we see that given $\sigma$ that fixes $P=P_{(a,b)}$, there exists $\gamma \in \mathbb{F}_{q^4}$ such that $\sigma(t)=\gamma t/\tilde{w},$ where $\gamma^m \in \mathbb{F}_q^*.$
However, since $\sigma$ is assumed to be the $\mathbb{F}_q$-rational lift of an element of $\Aut(\mathcal S_q)$, we may conclude that $\gamma \in \mathbb{F}_q^*.$

The value of $k$ such that $\sigma \tau^k$ fixes all points $\tilde{P} \in O_0$ of $\tilde{\mathcal S}_q$ lying above $P$ satisfies $\gamma \lambda^k=\tilde{w}(a,b).$ We claim that $i:=\gcd(k,m)=1$.
If this is not the case, then $(\sigma \tau^k)^{m/i}=\sigma^{m/i} \neq \id$ would fix all points $\tilde{P}$, but this is not possible, since $\iota(\sigma^{m/i})=0$ according to Theorem 26 in \cite{GMQZ}
(or see item five in Theorem \ref{theorem_26}).
Hence $\overline{\tau}:=\tau^k$ is a generator of $\Sigma_-$ and by construction $\sigma \overline{\tau}$ fixes all $\tilde{P} \in O_0$.
Redefining $\tau$ as $\overline{\tau}$ and $\lambda$ as $\lambda^k$, we obtain that $\gamma \lambda=\tilde{w}(a,b).$ Now let $k_d$ for $d=1,2,3$ satisfy $\gamma \lambda^{k_d}=\tilde{w}(a,b)^{q^d}.$ Then $\sigma \tau^{k^d}$
fixes orbit $O_d$ point-wise, but no other points. We then obtain
\[\lambda^{k_d}=\gamma^{-1}\gamma \lambda^{k_d}=\gamma^{-1}\tilde{w}(a,b)^{q^d}=\gamma^{-1}(\gamma \lambda)^{q^d}=\gamma^{q^d-1}\lambda^{q^d}=\lambda^{q^d}, \ \text{ for $d=1,2,3$,}\]
where in the last equality we used that $\gamma^{q-1}=1,$ since $\gamma \in \mathbb{F}_q^*$.
Since $\lambda$ has multiplicative order $m$, the proposition follows.
\end{proof}

\begin{remark}\label{rem:genusok}
The mistake in \cite{GMQZ} was that there it was claimed that if $\ord(\sigma) \mid (q - 2q_0 + 1)$, then $\iota(\sigma \tau^j) = 4m$ for exactly one $j \in \{ 1, \dots, m-1 \}$ and $\iota(\sigma \tau^j) = 0$ for all other $j \in \{0,1, \dots, m-1 \}$. However, in \cite{GMQZ} in this particular situation only subgroups were considered containing either all four or none of the possible elements with $\iota(\sigma \tau^j) = m$. Therefore the corresponding contribution to the different $\Delta_H$ was correctly taken to be $4m$ or $0$, implying that as far as genus computations are concerned, all the results obtained in \cite{GMQZ} are correct. In particular the results from Propositions 38, 39, 40, and 42 in \cite{GMQZ} are correct.
\end{remark}

The fact that in Theorem \ref{theorem_26} only the order of $\sigma$ is important, makes the last part of Theorem \ref{theorem:classification_subgroups_Sz} particularly useful in combination with Lemma \ref{lem:max}:
any subgroup of $\Sz(q)$ that is not contained in the first four listed maximal subgroups, is isomorphic to $\Sz(\hat{q})$ and as far as genus computations are involved only the isomorphism class matters.
Therefore, we can take the natural $\Sz(\hat{q}) \subseteq \Sz(q)$ obtained by restricting the field to $\mathbb{F}_{\hat{q}}.$
Combining this with Lemma \ref{lem:max} and the cases already treated in \cite{GMQZ}, we see that in order to deal with all subgroups of $\Sz(q) \times C_m$, we only need to deal with subgroups of
$B_0 \times C_m$ and $N_- \times C_m$. We will start with the second case.
Since $N_-$ contains an element of order $m$, we will also need to consider subgroups of a group of the form $C_m \times C_m$, which we deal with in the next subsection first.

\subsection{Description of subgroups of $C_m \times C_m$}

In this subsection we give a for us convenient description of all subgroups of the direct product $C_m \times C_m.$ The results in this subsection are valid for any value of $m$. We denote by $\sigma$ and $\tau$ two elements of $C_m \times C_m$ such that $\langle \sigma, \tau \rangle = C_m \times C_m$. Note that necessarily $\ord(\sigma)=\ord(\tau)=m.$ We start by describing a convenient set of generators of a subgroup of $C_m \times C_m$.

\begin{lemma}\label{lem:triple}
Let $m\ge 1$ be an integer, $C_m$ a cyclic group of order $m$, and $H$ a subgroup of $C_m \times C_m$. For each $H$, there exist unique positive integers $n_1$ and $n_2$ and a nonnegative integer $a$ such that:
\begin{enumerate}
\item $n_1|m$ and $n_2|m$,
\item $0 \le a < n_2$ and $n_1n_2|am$,
\item $H=\langle \sigma^{n_1}\tau^a,\tau^{n_2}\rangle$.
\end{enumerate}
\end{lemma}
\begin{proof}
Let $H \subseteq C_m \times C_m$ be a subgroup. Define $n_1$ to be the smallest positive integer for which there exists an integer $a$ such that $\sigma^{n_1}\tau^a \in H$. Now let $\sigma^c\tau^d \in H$, for certain integers $c$ and $d$. If $c=sn_1+r$, with $r,s \in \Z$ and $0 \le r < n_1$, then $\sigma^{c-sn_1}\tau^{d-sa} \in H$. The definition of $n_1$ implies that $r=c-sn_1=0$ and hence that $n_1$ divides $c$. In particular $n_1$ divides $m$, since $\sigma^m=\id \in H.$ Further let $n_2$ be the smallest positive integer such that $\tau^{n_2} \in H.$ Note that similar to $n_1$, the integer $n_2$ divides any $d$ for which $\tau^d \in H$ and in particular, $n_2$ divides $m$. Moreover, from their definitions, we see that both $n_1$ and $n_2$ are uniquely determined once $H$ is specified.

Multiplying $\sigma^{n_1}\tau^a$ with a suitable power of $\tau^{n_2}$, we may assume that $0 \le a < n_2$. The exponent $a$ thus obtained is unique, since if $\sigma^{n_1}\tau^a$ and $\sigma^{n_1}\tau^{a'}$ both are in $H$, then $\tau^{a-a'} \in H$, implying that $a \equiv a' \pmod{n_2}.$ This implies that either $a=a'$ or that $a' \ge n_2$. Note that since $(\sigma^{n_1}\tau^a)^{m/n_1}=\tau^{am/n_1} \in H$, we also obtain that $am/n_1$ is a multiple of $n_2$. All that remains to be shown is that $H=\langle \sigma^{n_1}\tau^a,\tau^{n_2}\rangle$. However, if $\sigma^c\tau^d \in H$, then we have seen that $n_1$ divides $c$. Hence for a suitably chosen integer $i$, we have $(\sigma^c\tau^d)(\sigma^{n_1}\tau^a)^i=\tau^{d+ia} \in H$. But then $n_2$ divides $d+ia$, implying that $\sigma^c\tau^d \in \langle \sigma^{n_1}\tau^a,\tau^{n_2}\rangle.$
\end{proof}
The uniqueness part of the previous lemma justifies the following definition.
\begin{definition}
Let $H \subset C_m \times C_m$ be a subgroup and $n_1,n_2,a$ be as in Lemma \ref{lem:triple}. We call the triple $(n_1,n_2,a)$ the standard exponents of $H$.
\end{definition}

As in the theory of finitely generated $\Z$-modules, one can simplify the description of $H$ even further if one can replace the generators $\sigma$ and $\tau$ with other generators of $C_m \times C_m$. This would result in an even simpler description where $a$ is equal to zero and $n_1$ divides $n_2$. However, as Theorem \ref{theorem_26} shows, the roles of $\sigma$ and $\tau$ are quite different, which is we have less freedom. Note that $|H|=m^2/(n_1n_2)$, since the elements of $H$ all can uniquely be written in the form $(\sigma^{n_1}\tau^a)^i(\tau^{n_2})^j$ with $0 \le i < m/n_1$ and $0 \le j < m/n_2.$


\subsection{Subgroups of $N_- \times C_m$}

As before, let $m=q-2q_0+1$ and consider the maximal subgroup $N_- \times C_m$ of $\Aut(\tilde{\mathcal{S}}_q)$, where $N_-$ is the normalizer of $\Sigma_-$ in $\Sz(q)$.
The group $N_-$ has order $4m$ and is isomorphic to $C_m \rtimes C_4$, where the semidirect product is defined by the homomorphism $\varphi : C_4 \to \Aut(C_m)$ mapping $\zeta$, a fixed generator of $C_4$,
to the automorphism $\omega \mapsto \zeta \omega \zeta^{-1} = \omega^{q}$. See \cite{HB}, Theorem 3.10, Chapter XI for details.
The group $N_- \times C_m$ is therefore isomorphic to $(C_m \rtimes C_4) \times C_m$ and can be presented as
$$ \langle \zeta, \sigma, \tau \mid \ord(\zeta) = 4, \ord(\sigma) = \ord(\tau) = m, \zeta \sigma \zeta^{-1} = \sigma^{q}, \zeta \tau = \tau \zeta, \sigma \tau = \tau \sigma \rangle. $$
It is easy to see that all elements of order two in $N_- \times C_m$ are those of the form $\sigma^i\zeta^2$, while the elements of order four are those of the form $\sigma^i\zeta$ or $\sigma^i\zeta^3 $.
Finally, $N_-$ has a maximal subgroup $D_-$ isomorphic to the dihedral group of order $2m$, containing $\Sigma_-$.
The group $D_- \times C_m = \langle \zeta^2, \sigma, \tau \rangle$ is isomorphic to $(C_m \rtimes C_2) \times C_m$.

The genera of the quotient curves $\tilde{\mathcal{S}}_q/H$ for $H$ subgroup of $N_- \times C_m$ of the form $(C_{n_1} \rtimes C_4)\times C_{n_2}$, $(C_{n_1} \rtimes C_2)\times C_{n_2}$, and $C_{n_1}\times C_{n_2}$,
with $n_1$ and $n_2$ divisors of $m$,  were computed in Propositions 38, 39, 40 from \cite{GMQZ}. To find which subgroups of $N_- \times C_m$ are missing, we first give the following proposition.

\begin{proposition}\label{prop:subgroupsNminxCm}
Let $H$ be a subgroup of $N_- \times C_m$. Then there exist divisors $n_1$ and $n_2$ of $m$ such that one of the following holds:
\begin{enumerate}
\item $H \subseteq \Sigma_- \times C_m$,
\item $H$ is conjugated to $\langle \sigma^{m/n_1},\tau^{m/n_2},\zeta^2 \rangle \cong (C_{n_1} \rtimes C_2)\times C_{n_2}$, or
\item $H$ is conjugated to $\langle \sigma^{m/n_1},\tau^{m/n_2},\zeta \rangle \cong (C_{n_1} \rtimes C_4)\times C_{n_2}$.
\end{enumerate}
\end{proposition}
\begin{proof}
It is clear that the subgroup $\Sigma_- \times C_m$ is normal in $N_-\times C_m$ and that the quotient group is cyclic of order four. Now let $H$ be a subgroup of $N_-\times C_m$ and consider the group homomorphism
$\phi: H \to (N_-\times C_m)/(\Sigma_-\times C_m)$. Then $\ker(\phi)=H \cap (\Sigma_-\times C_m)$ and therefore $H/(H \cap (\Sigma_-\times C_m))$ is a subgroup of a cyclic group of order four. Therefore the index of $H \cap (\Sigma_-\times C_m)$
in $H$ is in $\{1,2,4\}$.
Since $H \cap (\Sigma_-\times C_m)$ itself has odd cardinality, the Schur--Zassenhaus theorem implies that $H \cap (\Sigma_-\times C_m)$ has a complement $K$ in $H$. Moreover, $K$ is isomorphic to $H/(H \cap (\Sigma_-\times C_m))$.
We now distinguish three cases.

Case 1, $[H:(H \cap (\Sigma_-\times C_m))]=1$. In this case $H=H \cap (\Sigma_-\times C_m)$ and hence $H \subseteq \Sigma_- \times C_m.$

Case 2, $[H:(H \cap (\Sigma_-\times C_m))]=2$. In this case the complement $K$ contains an element of order two and $H$ cannot contain an element of order four.
Moreover, $\sigma^{-j}(\sigma^i \zeta^2)\sigma^{j}=\sigma^{i+j(q^2-1)}\zeta^2.$ Since $\gcd(q^2-1,m)=1$, we can choose $j$ such that $\sigma^{-j}(\sigma^i \zeta^2)\sigma^{j}=\zeta^2$.
Hence replacing $H$ by a suitable conjugate, we may assume that $\zeta^2$ is an element of $H$.
If $\sigma^i\tau^j\zeta^2\in H$,
then $\sigma^i\tau^j \in H$.
Moreover, if $\sigma^i\tau^j \in H$, then
\[H \ni \zeta^2(\sigma^i\tau^j)\zeta^2(\sigma^i\tau^j)^{-1}=\zeta^2\sigma^{i}\zeta^2\sigma^{-i}=\sigma^{i(q^2-1)}.\]
Since $\gcd(q^2-1,m)=1$ and $\ord(\sigma)=m$, this implies that $\sigma^i \in H.$ Hence whenever $\sigma^i\tau^j\zeta^e \in H$, with $e=0,2$,
then $\sigma^i \in H$ and hence $\tau^j \in H.$
This shows that $H$ is of the form as in case two of the proposition.

Case 3, $[H:(H \cap (\Sigma_-\times C_m))]=4$. This case is very similar to the previous one. After a suitable conjugation, we may assume that $\zeta \in H$. Then if $\sigma^i\tau^j\zeta^e \in H$, with $e=0,1,2,3$,
practically the same computation as before shows that
$\sigma^i \in H$ and $\tau^j \in H$. Hence $H$ is of the form as in case three of the proposition.
\end{proof}

This proposition shows that in order to complete the case of subgroups of $N_- \times C_m$, we only need to consider subgroups of $\Sigma_- \times C_m$. Therefore, we now turn our attention to those.
It will be convenient to define $v_p(n):=\max \{e \,:\, p^e \text{ divides } n\},$ where $n\in \Z$.

\begin{theorem}
  \label{theorem:genus_cm_times_cm_Sz}
  Let $H$ be a subgroup of $\Sigma_- \times C_m =\langle \sigma,\tau \rangle$ with standard exponents $(n_1,n_2,a)$. Suppose that $m=p_1^{e_1}\cdots p_r^{e_r}$, with $p_1,\dots,p_r$ mutually distinct prime numbers and $e_1,\dots,e_r$ positive integers.
  For $d \in \{0,1,2,3\}$, write $\nu_{d,\ell}=\min\{v_{p_\ell}(n_1q^d-a),v_{p_\ell}(n_2)\}.$
  The genus of the quotient curve $\tilde{\mathcal{S}}_q/H$ is
  $$ g_H = \frac{(q^2+1)(q-2) - \Delta_H}{2 |H|} + 1, $$
  with $|H| = m^2/(n_1n_2)$ and 
  $$ \Delta_H = \left( \frac{m}{n_2} - 1 \right) \cdot (q^2 + 1) + \sum_{d=0}^3 \left( \frac{m\prod_{\ell=1}^r p_\ell^{\nu_{d,\ell}}}{n_1n_2} - 1 \right) \cdot m. $$
\end{theorem}
\begin{proof}
The expression for the genus $g_H$ follows from the Riemann--Hurwitz formula and it has already been noted that $|H| = m^2/(n_1n_2)$. Therefore, all that remains to be done is to compute the quantity $\Delta_H$.
Recall that
$$ \Delta_H = \sum_{\substack{\omega \in H \\ \omega \neq \id}} \iota(\omega) = \sum_{\substack{\alpha > 0, \\ \sigma^\alpha \in H}} \iota(\sigma^\alpha) + \sum_{\substack{\beta > 0, \\ \tau^\beta \in H}} \iota(\tau^\beta) + \sum_{\substack{\alpha > 0, \beta > 0, \\ \sigma^\alpha \tau^\beta \in H}} \iota(\sigma^\alpha \tau^\beta)$$
and that $\iota(\sigma^\alpha) = 0$ for all $\alpha = 1, \dots, m-1$ and $\iota(\tau^\beta) = q^2 + 1$ for all $\beta = 1, \dots, m-1$ from Theorem \ref{theorem_26}. \\

The elements of $H$ can uniquely be written in the form $(\sigma^{n_1}\tau^a)^i(\tau^{n_2})^j$ with $0 \le i < m/n_1$ and $0 \le j < m/n_2.$ Hence the number of elements in $H$ of the form $\tau^\beta \neq \id$ is exactly $m/n_2-1$. What remains to be done is to determine the number of elements in $H \setminus \{\id\}$ of the form $\sigma^\alpha\tau^\beta$ such that $\iota(\sigma^\alpha\tau^\beta)=m$.
From Theorem \ref{theorem_26} and Proposition \ref{prop:fourjs}, we conclude that this is equal to the number of triples $(i,j,d)$, with $d=0,1,2,3$, $1 \le i < m/n_1$, and $0 \le j < m/n_2$, such that
\begin{equation}\label{eq:congruence}
jn_2\equiv i(n_1q^d-a) \pmod{m}.
\end{equation}
We will first for each $d$ count the number of solutions $(i,j)$ to congruence \eqref{eq:congruence} satisfying $0 \le i < m$ and $0 \le j < m$. In order to do this, we use the factorization of $m$ into prime numbers: $m=p_1^{e_1}\cdots p_r^{e_r}$. The congruence $jn_2\equiv i(n_1q^d-a) \pmod{p_\ell^{e_\ell}}$ is equivalent to the congruence $j(n_2/p_ \ell^{\nu_{d,\ell}})\equiv i(n_1q^d-a)/p_\ell^{\nu_{d,\ell}} \pmod{p_\ell^{e_\ell-\nu_{d,\ell}}}.$ By definition of $\nu_{d,\ell}$, both $n_2/p_ \ell^{\nu_{d,\ell}}$ and $(n_1q^d-a)/p_\ell^{\nu_{d,\ell}}$ are integers and at least one of them is not divisible by $p_\ell$ and therefore has an inverse modulo powers of $p_\ell$. This means that either $i$ or $j$ can be chosen arbitrarily between $0$ and $p_\ell^{e_\ell}-1$, while the other variable then is determined uniquely modulo $p_\ell^{e_\ell-\nu_{d,\ell}}.$ This means that the congruence $jn_2\equiv i(n_1q^d-a) \pmod{p_\ell^{e_\ell}}$ has exactly $p_\ell^{e_\ell+\nu_{d,\ell}}$ many solutions $(i,j)$ with $0 \le i < p_\ell^{e_\ell}$ and $0 \le j < p_\ell^{e_\ell}$. Using the Chinese remainder theorem, we see that congruence \eqref{eq:congruence} for a given $d$ has exactly $\prod_\ell p_\ell^{e_\ell+\nu_{d,\ell}}=m\prod_\ell p_\ell^{\nu_{d,\ell}}$ many solutions $(i,j)$ satisfying $0 \le i < m$, and $0 \le j < m$.

Now note that if $(i,j)$ is such a solution, then for any integers $\alpha$ and $\beta$, the pair $(i+m/n_1\alpha \bmod{m},j+m/n_2\beta-ma/(n_1n_2)\alpha \bmod{m})$ also is a solution (note that $n_1n_2$ divides $ma$ by Lemma \ref{lem:triple}). This means that any solution $(i,j)$ to congruence \eqref{eq:congruence} gives rise to $n_1n_2$ solutions when $\alpha$ and $\beta$ are chosen such that $0 \le \alpha < m/n_1$ and $0 \le \beta < m/n_2.$ Moreover any such a set of $n_1n_2$ solutions contains exactly one solution pair $(i,j)$ satisfying $0 \le i < m/n_1$ and $0 \le j < m/n_2$. We may therefore conclude that the number of solutions $(i,j)$ to congruence \eqref{eq:congruence} satisfying $0 \le i < m/n_1$ and $0 \le j < m/n_2$ is equal to $m\prod_\ell p_\ell^{\nu_{d,\ell}}/(n_1n_2)$. Disregarding the solution $(0,0)$, we conclude that the number of elements $h \in H\setminus\{\id\}$ for which $\iota(h)=m$ equals $m\prod_\ell p_\ell^{\nu_{d,\ell}}/(n_1n_2)-1$. The result now follows.
\end{proof}

This completes the study of the genera of the quotient curves $\tilde{\mathcal{S}}_q/H$ for $H$ subgroup of $N_- \times C_m$.
The following genera, obtained using Theorem \ref{theorem:genus_cm_times_cm_Sz} for $s =1, 2, 3, 4$, are new to the best of our knowledge.

\begin{table}[h]
  \centering
  \begin{tabularx}{\textwidth}{llX}
    \hline
    s & Field & Genus \\
    \hline
    1 & $\mathbb{F}_{2^{12}}$ & 38 \\
    2 & $\mathbb{F}_{2^{20}}$ & 104, 534, 604, 614, 3066 \\
    3 & $\mathbb{F}_{2^{28}}$ & 9080 \\
    4 & $\mathbb{F}_{2^{36}}$ & 3484, 10420, 129160, 135688, 138736, 138952, 138958, 138970, 1806442, 5141854\\
    \hline
  \end{tabularx}
  \caption{New genera from Theorem \ref{theorem:genus_cm_times_cm_Sz} for $s=1, 2, 3, 4$.}
  \label{tab:cm_times_cm_Sz}
\end{table}

\subsection{Subgroups of $B_0 \times C_m$}

Let $B_0$ be the maximal subgroup of $\Sz(q)$ of order $2(q-1)$, isomorphic to the dihedral group $D_{q-1}$, corresponding to the second case in Theorem \ref{theorem:classification_subgroups_Sz}.
Since $q-1$ and $m$ are coprime, all subgroups of $B_0 \times C_m$ are either of the form $C_d \times C_n$ or of the form $D_d \times C_n$, with $d$ dividing $q-1$ and $n$ dividing $m$. Conversely, for each $d$
dividing $q-1$ and $n$ dividing $m$ there exists a subgroup of $B_0 \times C_m$ isomorphic to $C_d \times C_n$ and a subgroup of $B_0 \times C_m$ isomorphic to $D_d \times C_n$.

\begin{theorem}
  \label{theorem:delta_b0_times_cm}
  Let $H$ be a subgroup of $B_0 \times C_m$. If $H \simeq C_d \times C_n$ for some $d$ dividing $q-1$ and $n$ dividing $m$, then the genus of the quotient curve $\tilde{\mathcal{S}}_q/H$ is
  $$ g_H = \frac{(q^2+1)(q-n-1) - 2(d - 1)n}{2dn} + 1. $$
  If $H \simeq D_d \times C_n$ for some $d$ dividing $q-1$ and $n$ dividing $m$, then the genus of the quotient curve $\tilde{\mathcal{S}}_q/H$ is
  $$ g_H = \frac{(q^2+1)(q-n-1) - dm(2q_0+1) - 3dn + 2n}{4dn} + 1. $$
\end{theorem}

\begin{proof}
  If $H \simeq C_d \times C_n$, then $\Delta_H = (n - 1) \cdot (q^2 + 1) + (d - 1)n \cdot 2$ from Theorem \ref{theorem_26}. The expression for $g_H$ follows from the Riemann--Hurwitz formula. \\
  Assume $H \simeq D_d \times C_n$ now. Let $\mathfrak{s}$ be an element of $D_d$ of order $2$. As a set, $H$ can be expressed as disjoint union of subsets as follows:
  $$ H = \left( C_d \times C_n \right) \cup \left( \mathfrak{s} C_d \times \{ \id_{C_n} \} \right) \cup \left( \mathfrak{s} C_d \times \left( C_n \setminus \{ \id_{C_n} \} \right) \right). $$
  From the previous case and from Theorem \ref{theorem_26}, $\Delta_H$ can be obtained as
  $$ \Delta_H = (n - 1) \cdot (q^2 + 1) + (d - 1)n \cdot 2 + d \cdot (m(2q_0+1)+1) + d(n-1) \cdot 1. $$
  The conclusion follows from the Riemann--Hurwitz formula.
\end{proof}

No new genera are found for $s = 1, 2, 3, 4$ using Theorem \ref{theorem:delta_b0_times_cm}.

\section{Galois subcovers of $\tilde{R}_q$}

Whereas we in the previous section studied Galois subcovers $\tilde{\mathcal{S}}_q$ of the form $\tilde{\mathcal{S}}_q/H$, we now deal with the case of the Ree curve and the associated Skabelund curve $\tilde{\mathcal R}_q$.
Throughout this section let $s \geq 1$ be a fixed integer, $q_0 := 3^s$, $q := 3q_0^2$ and $m := q - 3q_0 + 1$. In this subsection, $\tau$ denotes the automorphism of $\tilde{\mathcal R}_q$ fixing $x$,$y$, and $z$, while
mapping $t$ to $\lambda t$, with $\lambda \in \mathbb{F}_{q^6}$ an element of multiplicative order $m.$ It will convenient to define ten functions $w_i$, $i=1,\dots,10$, on the Ree curve that were introduced in \cite{Pedersen}.

\bigskip
\noindent
$
\begin{array}{llll}
w_1:=x^{3q_0+1}-y^{3q_0} &  w_2:=xy^{3q_0}-z^{3q_0} &  w_3:=xz^{3q_0}-w_1^{3q_0}\\
w_4:=xw_2^{q_0}-yw_1^{q_0} &   v:=xw_3^{q_0}-zw_1^{q_0} &  w_5:=yw_3^{q_0}-zw_2^{q_0}\\
w_6:=v^{3q_0}-w_2^{3q_0}+xw_4^{3q_0} &  w_7:=yw_2^{q_0}-xw_3^{q_0}-w_6^{3q_0} &  w_8:=w_5^{3q_0}+xw_7^{3q_0}\\
w_9:=w_2^{q_0}w_4-yw_6^{q_0} &  w_{10}:=zw_6^{q_0}-w_3^{q_0}w_4 &  \\
\end{array}
$

\bigskip
\noindent
These functions were used in \cite{DE} to obtain a smooth embedding of the Ree curve in thirteen-dimensional projective space. For future reference, we collect some facts on the function $w_8$ in the form of two lemmas.
\begin{lemma}\label{lem:divisorw8}
Let $P_{(0,0,0)}$, respectively $P_\infty$, be the common zero, respectively the only pole of the functions $x$, $y$ and $z$ on the Ree curve. Then
$$(w_8)=(q+1)(q+3q_0+1)(P_{(0,0,0)}- P_\infty)=\frac{q^3+1}{m}(P_{(0,0,0)}- P_\infty).$$
\end{lemma}
\begin{proof}
From \cite[Equation (A18)]{Pedersen} $v_{P_\infty}(w_8)=-(q+1)(q+3q_0+1)$ and $P_\infty$ is the only pole of $w_8$ as it is a polynomial in $x$,$y$ and $z$.
The lemma follows by proving that $v_{P_{(0,0,0)}}(w_8)=(q+1)(q+3q_0+1)$.
However, this follows directly from the defining equations for the functions $w_i$ and $v$ using the strict triangle equality and the fact that
$v_{P_{(0,0,0)}}(x)=1$, $v_{P_{(0,0,0)}}(y)=q_0+1$, and $v_{P_{(0,0,0)}}(z)=2q_0+1$, which in turn can be deduced from the defining equation of $\mathcal R_q$.
\end{proof}

\begin{lemma}\label{lem:zeroesw8qminusw8}
All zeroes of the function $w_8^q-w_8$ on the curve $\mathcal R_q$ are $\mathbb{F}_q$-rational.
\end{lemma}

\begin{proof}
We can use the definitions of the functions $w_i$ recalled above and the relations between them given in \cite[Appendix A]{Pedersen} to find a contradiction assuming that a zero $P$ of $w_8^q-w_8$ which
is not $\mathbb{F}_q$-rational exists. From \cite[Equation (A18)]{Pedersen} $w_8^q-w_8=w_7^{3q_0}(x^q-x)$ so $P$ is also a zero of $w_7$. This implies together with the definition of $w_8$ that $w_5$ also vanishes at $P$.
From \cite[Equation (A16)]{Pedersen} and the defining equations of $\mathcal R_q$,
$w_7^q-w_7=w_2^{q_0}(y^q-y)-w_3^{q_0}(x^q-x)=(w_2x-w_3)^{q_0}(x^q-x)$, so that $P$ is also a zero of $w_2x-w_3$.
 This shows that $P$ is a common zero of $w_2x-w_3$, $w_5$, and $w_7$. From the definition of $w_5$ we have
$$0=w_5(P)=(yw_3^{q_0}-zw_2^{q_0})(P)=(yx^{q_0}w_2^{q_0}-zw_2^{q_0})(P)=w_2^{q_0}(P) \cdot (yx^{q_0}-z)(P).$$
From \cite[Equation (A4)]{Pedersen} $w_2(P) \ne 0$ since the zeros of $w_2^q-w_2=y^{3q_0}(x^q-x)$, and hence of $w_2$, are all $\mathbb{F}_q$-rational. Hence $P$ needs to be a zero of $yx^{q_0}-z$.
Combining this with $z^q-z=x^{q_0}(y^q-y)$ we get that
$$x^{q_0}(y^q-y)(P)=z(P)^q-z(P)=y^qx^{q_0q}(P)-yx^{q_0}(P),$$
hence $x^{q_0}y^q(P)=x^{q_0q}y^q(P)$ and $y^q(x^q-x)^{q_0}(P)=0$. This implies that $P$ is $\mathbb{F}_q$-rational, a contradiction.
\end{proof}

We now follow exactly the same approach as in the previous section and start with the Ree-variant of Proposition \ref{prop:rationalliftSuzuki}. It refines Lemma 4.2 from \cite{Sk}.

\begin{proposition}\label{prop:rationalliftRee}
Every automorphism of $\mathcal{R}_q$ can be lifted to an automorphism of $\tilde{\mathcal R}_q$ defined over $\mathbb{F}_q$ in a unique way.
The resulting collection of automorphisms forms a group isomorphic to $\Ree(q).$
\end{proposition}
\begin{proof}
The automorphism group $\Aut(\mathcal R_q)$ is generated by an involution $\phi$ and automorphisms $\psi_{abcd}$ defined by $\psi_{abcd}(x)=ax+b$, $\psi_{abcd}(y)=a^{q_0+1}y+ab^{q_0}x+c$, and $\psi_{abcd}(z)=a^{2q_0+1}z-a^{q_0+1}b^{q_0}y+ab^{2q_0}x+d$, with $a \in \mathbb{F}_q^*$ and $b,c,d \in \mathbb{F}_q.$
As explained in \cite{Sk}, $\psi_{abcd}$ can be lifted to an automorphism $\psi$ of $\tilde{\mathcal R}_q$ by setting $\psi(t)=\alpha t$, where $\alpha^m=a.$ Since $\gcd(q-1,m)=1$,
there exists exactly one choice for $\alpha \in \mathbb{F}_q^*$ such that $\alpha^m=a$.
The involution $\phi$, satisfies $\phi(x)=w_6/w_8$, $\phi(y)=w_{10}/w_8$, and $\phi(x)=w_9/w_8$. As observed in \cite{Sk}, it can be lifted to an automorphism of $\tilde{\mathcal R}_q$ by defining $\phi(t)=t/w_8.$
The remainder of the proof is now similar as the proof of Proposition \ref{prop:rationalliftSuzuki}.
\end{proof}

As for the Suzuki case, we will call the lift of $\sigma \in \Aut(\mathcal{R}_q)$ described in Proposition \ref{prop:rationalliftRee} the $\mathbb{F}_q$-rational lift of $\sigma$ and denote this lift again by $\sigma$.
Also the fact already proved in \cite{GMQZ} that $\Aut(\tilde{\mathcal{R}}_q)=\Ree(q) \times C_m$ is again an easy consequence of the existence of $\mathbb{F}_q$-rational lifts. Continuing the same strategy as in the previous section,
we now collect various facts on subgroups of $\Ree(q)$. See \cite{LN}, Section 2 for details. 
\begin{theorem}
  \label{theorem:classification_subgroups_Ree}
  Up to conjugation, the Ree group $\Ree(q)$ has the following maximal subgroups.
  \begin{enumerate}
  \item The Frobenius group $F$ of order $q^3(q-1)$.
  \item The centralizer $C$ of an involution, of order $q(q-1)(q+1)$.
  \item The normalizer $N_-$ of a cyclic Singer group $\Sigma_-$ with $|\Sigma_-| = q-3q_0+1$ and $|N_-| = 6 \cdot |\Sigma_-|$.
  \item The normalizer $N_+$ of a cyclic Singer group $\Sigma_+$ with $|\Sigma_+| = q+3q_0+1$ and $|N_+| = 6 \cdot |\Sigma_+|$.
  \item The normalizer $N$ of a cyclic group $A$ with $|A| = (q+1)/4$ and $|N| = 6(q+1)$.
  \item The Ree groups $\Ree(\hat{q})$ for $q = \hat{q}^h$, with $h$ a prime number.
  \end{enumerate}
  Let $N_2$ be the normalizer of a Sylow 2-subgroup of $\Ree(q)$. Then, for any subgroup $K$ of $\Ree(q)$, one of the following three possibilities
  occurs: $K$ is isomorphic to $\Ree(\hat{q})$ where $q=\hat{q}^k$ and $1 \le k \le 2s+1$, or $K$ is isomorphic to $\PSL(2, 8)$,
  or $K$ is conjugated to a subgroup of one of $F$, $C$, $N_-$, $N_+$, $N$, or $N_2$. Finally, the subgroup $N_2$ has order $168$, $\PSL(2, 8)$ has
  order $504$ and both have conjugates contained in $\Ree(3)$.
\end{theorem}

Lemma \ref{lem:max} and Theorem \ref{theorem:classification_subgroups_Ree} allow us to describe all maximal subgroups of $\Aut(\tilde{\mathcal{S}}_q)=\Ree(q)\times C_m$.
We obtain the following analogue of Corollary \ref{cor:subgroupsautskabelund}:
\begin{corollary}\label{cor:subgroupsautskabelund2}
Any subgroup $H \subset \Ree(q)\times C_m$ is either of the form $\Ree(q) \times C_n$, with $n|m$ and $C_n \subseteq C_m$ the unique subgroup of order $n$, or contained in $M \times C_m$ with $M$ a maximal subgroup of $\Ree(q).$
\end{corollary}
\begin{proof}
The proof is similar to that of Corollary \ref{cor:subgroupsautskabelund}. The main ingredient is that Theorem \ref{theorem:classification_subgroups_Ree} implies that a subgroup of $\Ree(q)$ of index at most $m$,
is equal to $\Ree(q)$ itself.
\end{proof}

In \cite{GMQZ}, the genus of the quotient curve $\tilde{\mathcal{R}}_q/H$ is computed when $H$ is one of following subgroups of $\Aut(\tilde{\mathcal{R}}_q)$:
\begin{itemize}
\item $F \times C_m$ or one of its subgroups;
\item $C \times C_m$ or one of its subgroups;
\item $N_+ \times C_m$ or one of its subgroups;
\item $N_- \times C_m$ or one of its subgroups of the form $K \times C_n$ with $K$ a subgroup of $N_-$ and $n$ dividing $m$;
\item $N \times C_m$ or one of its subgroups;
\item $\Ree(\hat{q}) \times C_n$ for suitable $\hat{q}$ and for $n$ dividing $m$.
\end{itemize}

Corollary \ref{cor:subgroupsautskabelund2} implies that the only cases where the genus of $\tilde{\mathcal{R}}_q/H$ has not been computed yet are if $H$ is one of the missing subgroups of $N_- \times C_m$ or $H$ is one of the
missing subgroups of $\Ree(3) \times C_m$. For a generic subgroup $H$ of $\Aut(\tilde{\mathcal{R}}_q)$, let $g_H$ be the genus of the quotient curve $\tilde{\mathcal{R}}_q/H$. Again, we can use the theory of ramification groups
and Hilbert's different formula to compute $g_H$.
For each $\omega \in \Aut(\tilde{\mathcal{R}}_q)$, the quantity $\iota(\omega)$ was computed in \cite{GMQZ}, Theorem 48; however, as in the Suzuki case one mistake was made.
Hereby we give the correct formulation and include a proof for the corrected case.
%
\begin{theorem}\label{theorem_48}
Let $\sigma \in \Ree(q) \times \{ \id_{C_m} \}$, $\sigma \neq \id$ and $\{ \id_{\Ree(q)} \} \times C_m = \langle \tau \rangle$. Then $\iota(\tau^k) = q^3 + 1$ for all $k = 1, \dots, m-1$ and one of the following cases occurs.

\bigskip
\noindent
1. $\ord(\sigma) = 3$, $\sigma$ is in the center of a Sylow 3-subgroup, $\iota(\sigma) = m(q+3q_0+1)+1$, and $\iota(\sigma \tau^k) = 1$ for all $k = 1, \dots, m-1$;

\bigskip
\noindent
2. $\ord(\sigma) = 3$, $\sigma$ is not in the center of any Sylow 3-subgroup $\iota(\sigma) = m(3q_0+1)+1$, and $\iota(\sigma \tau^k) = 1$ for all $k = 1, \dots, m-1$;

\bigskip
\noindent
3. $\ord(\sigma) = 9$, $\iota(\sigma) = m+1$, and $\iota(\sigma \tau^k) = 1$ for all $k = 1, \dots, m-1$;

\bigskip
\noindent
4. $\ord(\sigma) = 2$, $\iota(\sigma \tau^k) = q+1$ for all $k = 0, \dots, m-1$;

\bigskip
\noindent
5. $\ord(\sigma) = 6$, $\iota(\sigma \tau^k) = 1$ for all $k = 0, \dots, m-1$;

\bigskip
\noindent
6. $\ord(\sigma) \mid (q-1)$, $\ord(\sigma) \neq 2$, $\iota(\sigma \tau^k) = 2$ for all $k = 0, \dots, m-1$;

\bigskip
\noindent
7. $\ord(\sigma) \mid (q+1)$, $\ord(\sigma) \neq 2$, $\iota(\sigma \tau^k) = 0$ for all $k = 0, \dots, m-1$;

\bigskip
\noindent
8. $\ord(\sigma) \mid (q + 3q_0 + 1)$, $\iota(\sigma \tau^k) = 0$ for all $k = 0, \dots, m-1$;

\bigskip
\noindent
9. $\ord(\sigma) \mid (q - 3q_0 + 1)$, $\iota(\sigma) = 0$, $\iota(\sigma \tau^j) = m$ for exactly six distinct $j \in \{ 1, \dots, m-1 \}$ and $\iota(\sigma \tau^j) = 0$ for all other $j$ between $1$ and $m-1$.
\end{theorem}

\begin{proof}
Only the statements about $\iota(\sigma \tau^j)$ for $j=1,\dots,m-1$ in the ninth item need a proof, the rest of the theorem being identical to \cite{GMQZ}, Theorem 48.

Let $\sigma \in \Sigma_- \setminus \{\id\}$. Then $\sigma$ fixes an $\mathbb{F}_{q^6}$-rational, not $\mathbb{F}_q$-rational, point $P$ of the Ree curve $\mathcal{R}_q$ with certain affine coordinates
$(x(P),y(P),z(P))=(a,b,c)$. Further, $\sigma$ maps $P_\infty$, the unique pole of $x$, to an $\mathbb{F}_q$ rational point, say $P_{(\alpha_1,\beta_1,\gamma_1)}$, having affine coordinates $(\alpha_1,\beta_1,\gamma_1).$
By Proposition \ref{prop:rationalliftRee}, we know that $\sigma \in \Ree(q)$ can be lifted uniquely to an element in $\Aut(\tilde{\mathcal R}_q)$ defined over $\mathbb{F}_q$, which we denote by $\sigma$ again for convenience.
First of all, using Lemma \ref{lem:divisorw8}, one shows similarly as in the proof of Theorem \ref{theorem_26} that
$$\left( \frac{\sigma(x^q+x)}{x^q+x}\right)_{\mathcal{R}_q}=(q^3+1)(P_\infty-P_{(\alpha_1,\beta_1,\gamma_1)})=\big( \tilde{w}^{-m}\big)_{\mathcal{R}_q}.$$
where $\tilde{w}:=\omega(w_8)$ and $\omega \in \Aut(\mathcal R_q)$ is an element such that $\omega(P_\infty)=P_\infty$ and $\omega(P_{(0,0,0)})=P_{(\alpha_1,\beta_1,\gamma_1)}.$ Note that $\omega$ exists, since $\Ree(q)$
acts $2$-transitive on the set of $\mathbb{F}_q$-rational points of $\mathcal R_q$.
We may conclude that $\sigma(t)=\gamma t/\tilde{w},$ for some $\gamma \in \mathbb{F}^*_{q^6}$ and that for all $k=0,\ldots,m-1$,
$$\sigma\tau^k(x)=\sigma(x), \ \sigma\tau^k(y)=\sigma(y), \ \sigma\tau^k(t)=\gamma \lambda^k \frac{t}{\tilde{w}},$$
where $\lambda \in \mathbb{F}^*_{q^6}$ is an element of multiplicative order $m$.

Now denote for $i=0,1,\dots,5$ by $P_i$ the point of $\mathcal R_q$ with affine coordinates $(a^{q^i},b^{q^i},c^{q^i})$ and let $O_i$ be the set of points lying above $P_{i}$ in the cover
$\tilde{\mathcal{R}}_q \to \mathcal{R}_q$. Note $P_0=P$. Similarly as in the Suzuki case, if $\tilde{P} \in O_0$ is a point lying above $P$ then there exists a unique $k$ between $1$ and $m-1$ such that
$\sigma \circ \tau^k$ fixes $\tilde{P}$ implying that $\gamma \lambda^k=\tilde{w}(a,b,c)$ and hence that $\sigma\tau^k$ fixes the orbit $O_0$ point-wise.
To show that none of the points in the remaining orbits $O_i$ are fixed by $\sigma\tau^k$ if equivalent to showing that $\tilde{w}(a,b,c) \not\in \mathbb{F}_{q^i}$.
Since for $i=1,\dots,5$, one has $\gcd(q^i-1,m)=1$, we see that $\tilde{w}(a,b,c) \in \mathbb{F}_{q^i}$ implies that $\tilde{w}(a,b,c) \in \mathbb{F}_{q}$.
However, by Lemma \ref{lem:zeroesw8qminusw8} this cannot occur.
We may conclude that $\iota(\sigma \tau^k)=m$. Starting with a point in one of the other orbits $O_i$, one can similarly find a unique $k$, a different one for each orbit, such that $\iota(\sigma \tau^k)=m$.
\end{proof}

As in the Suzuki case, we can supplement this with the following result:

\begin{proposition}\label{prop:sixjs}
Let $\sigma$ be the $\mathbb{F}_q$-rational lift of an element of $\Aut(\mathcal R_q)$ of order $q - 3q_0 + 1$.
Then there exists a choice of the generator $\tau$ of $C_m$ such that the six values of $j$ for which $\iota(\sigma \tau^j) = m$ are $q^d \bmod{m}$ for $d = 0,1,2,3,4,5$.
\end{proposition}
\begin{proof}
The proof is completely similar to that of Proposition \ref{prop:fourjs}.
\end{proof}

\begin{remark}\label{rem:genusok2}
As in Remark \ref{rem:genusok}, one can show that the formulation of \cite{GMQZ}, Theorem 48 does not affect the genus computations carried out
in \cite{GMQZ}. In particular, Propositions 65, 66, 67, 68, and 72 from \cite{GMQZ} are correct.
\end{remark}

We now consider the two cases not fully treated in \cite{GMQZ} in the following two subsections, namely all subgroups of $N_- \times C_m$ and all subgroups of $\Ree(3) \times C_m$.

\subsection{Subgroups of $N_- \times C_m$}

The normalizer $N_-$ of a Singer cycle $\Sigma_-$ has order $6m$ and is isomorphic to $C_m \rtimes C_6$, where the semidirect product is
defined by the homomorphism $\varphi : C_6 \to \Aut(C_m)$ mapping $\zeta$, a fixed generator of $C_6$,
to the automorphism $\omega \mapsto \zeta \omega \zeta^{-1} = \omega^{q}$. See \cite{HB}, Theorem 13.2, Chapter XI and \cite{CO}, Proposition 4.13 for details.
The group $N_- \times C_m$ is therefore isomorphic to $(C_m \rtimes C_6) \times C_m$ and can be presented as
$$\langle \zeta, \sigma, \tau \mid \ord(\zeta) = 6, \ord(\sigma) = \ord(\tau) = m, \zeta \sigma \zeta^{-1} = \sigma^{q}, \zeta \tau = \tau \zeta, \sigma \tau = \tau \sigma \rangle.$$

It is easy to see that all elements of order two in $N_- \times C_m$ are those of the form $\sigma^i\zeta^3$, while the elements of order three are those of the form $\sigma^i\zeta^2$ or $\sigma^i\zeta^4$.
Finally the elements of order six are those of the form $\sigma^i\zeta$ or $\sigma^i\zeta^5$.

To find out which subgroups of $N_- \times C_m$ have not been treated in \cite{GMQZ} yet, we give the following analogue of
Proposition \ref{prop:subgroupsNminxCm}.

\begin{proposition}\label{prop:subgroups2}
Let $H$ be a subgroup of $N_- \times C_m$. Then there exists divisors $n_1$ and $n_2$ of $m$ such that one of the following holds:
\begin{enumerate}
\item $H \subseteq \Sigma_- \times C_m$,
\item $H$ is conjugated to $\langle \sigma^{m/n_1},\tau^{m/n_2},\zeta^3 \rangle \cong (C_{n_1} \rtimes C_2)\times C_{n_2}$,
\item $H$ is conjugated to $\langle \sigma^{m/n_1},\tau^{m/n_2},\zeta^2 \rangle \cong (C_{n_1} \rtimes C_3)\times C_{n_2}$, or
\item $H$ is conjugated to $\langle \sigma^{m/n_1},\tau^{m/n_2},\zeta \rangle \cong (C_{n_1} \rtimes C_6)\times C_{n_2}$.
\end{enumerate}
\end{proposition}
\begin{proof}
Let $H$ be a subgroup of $N_-\times C_m$. Just as in the proof of Proposition \ref{prop:subgroupsNminxCm}, the Schur--Zassenhaus theorem
implies that $H \cap (\Sigma_-\times C_m)$ has a complement $K$ in $H$, which is isomorphic to $H/(H \cap (\Sigma_-\times C_m))$.
Then four cases can be distinguished and dealt with similarly as in Proposition \ref{prop:subgroupsNminxCm}.

Case 1, $[H:(H \cap (\Sigma_-\times C_m))]=1$. In this case $H=H \cap (\Sigma_-\times C_m)$ and hence $H \subseteq \Sigma_- \times C_m.$

Case 2, $[H:(H \cap (\Sigma_-\times C_m))]=2$. In this case the complement $K$ contains an element $\sigma^i \zeta^3$ of order two and
since $\sigma^{-j}(\sigma^i \zeta^3)\sigma^{j}=\sigma^{i+j(q^3-1)}\zeta^3.$ Since $\gcd(q^3-1,m)=1$, we can choose $j$ such that $\sigma^{-j}(\sigma^i \zeta^3)\sigma^{j}=\zeta^3$.
Hence replacing $H$ by a suitable conjugate, we may assume that $\zeta^3$ is an element of $H$.
If $\sigma^i\tau^j \in H$, then
\[H \ni \zeta^3(\sigma^i\tau^j)\zeta^3(\sigma^i\tau^j)^{-1}=\zeta^3\sigma^{i}\zeta^3\sigma^{-i}=\sigma^{i(q^3-1)}.\]
Since $\gcd(q^3-1,m)=1$ and $\ord(\sigma)=m$, this implies that $\sigma^i \in H.$ Hence whenever $\sigma^i\tau^j\zeta^e \in H$, with $e=0,3$,
then $\sigma^i \in H$ and hence $\tau^j \in H.$ This shows that $H$ is of the form as in case two of the proposition.

Case 3, $[H:(H \cap (\Sigma_-\times C_m))]=3$. Can be handled in a similar way.

Case 4, $[H:(H \cap (\Sigma_-\times C_m))]=6$. Can be handled in a similar way.
\end{proof}

This proposition shows that the only subgroups of $N_- \times C_m$ not considered in \cite{GMQZ}, are contained in $\Sigma_- \times C_m$.
In the following, we determine the genus of $\tilde{R}_q/H$ for all possible subgroups of $\Sigma_- \times C_m$.

\begin{theorem}
  \label{theorem:genus_cm_times_cm_Ree}
  Let $H$ be a subgroup of $\Sigma_- \times C_m = \langle \sigma,\tau \rangle$ with standard exponents $(n_1,n_2,a)$. Suppose that $m=p_1^{e_1}\cdots p_r^{e_r}$,
  with $p_1,\dots,p_r$ mutually distinct prime numbers and $e_1,\dots,e_r$ positive integers.
  For $d \in \{0,1,2,3,4,5\}$, write $\nu_{d,\ell}=\min\{v_{p_\ell}(n_1q^d-a),v_{p_\ell}(n_2)\}.$
  The genus of the quotient curve $\tilde{R}_q/H$ is
  $$ g_H = \frac{(q^3+1)(q-2) - \Delta_H}{2 |H|} + 1, $$
  with $|H| = m^2/(n_1n_2)$ and
  $$ \Delta_H = \left( \frac{m}{n_2} - 1 \right) \cdot (q^3 + 1) + \sum_{d=0}^5 \left( \frac{m\prod_{\ell=1}^r p_\ell^{\nu_{d,\ell}}}{n_1n_2} - 1 \right) \cdot m. $$
\end{theorem}
\begin{proof}
The proof is very similar to the proof of Theorem \ref{theorem:genus_cm_times_cm_Sz} and is therefore omitted.
\end{proof}


The following genus, obtained using Theorem \ref{theorem:genus_cm_times_cm_Ree} for $s = 1$, is new up to our knowledge.

\begin{table}[h]
  \centering
  \begin{tabularx}{\textwidth}{llX}
    \hline
    s & Field & Genus \\
    \hline
    1 & $\mathbb{F}_{3^{18}}$ & 12942 \\
    \hline
  \end{tabularx}
  \caption{New genera from Theorem \ref{theorem:genus_cm_times_cm_Ree} for $s=1$.}
  \label{tab:cm_times_cm_Ree}
\end{table}

\subsection{Subgroups of $\Ree(3) \times C_m$}

Up to conjugation, $\Ree(3)$ has four maximal subgroups: $\PSL(2, 8)$, of order $504$, the normalizer $N_2$ of a Sylow 2-subgroup, with $|N_2| = 168$,
the Frobenius group $\hat{F}$ of order $54$ and the subgroup $\hat{N}_+$ of order $42$ normalizing a cyclic Singer group of order $7$.
In the following, we compute the genera of quotient curves $\tilde{R}_q/H$ when $H$ is a subgroup of $\Ree(3) \times C_m$. Some special cases have been already covered in \cite{GMQZ}, precisely the cases of subgroups $H$ of
$\Ree(3) \times C_m$ of the form $H = K \times C_n$, with $K = \Ree(3)$, $K$ subgroup of $\hat{F}$ or $K$ subgroup of $\hat{N}_+$ and $C_n$ subgroup of $C_m$. \\
We first complete the study of subgroups of $\Ree(3) \times C_m$ of the form $H = K \times C_n$, with $K$ a subgroup of $\Ree(3)$ and $C_n$ subgroup of $C_m$; we call these subgroups \textit{non-skew} subgroups of $\Ree(3) \times C_m$. Then, we will consider subgroups of $\Ree(3) \times C_m$ that are not direct product of a subgroup of $\Ree(3)$ and a subgroup of $C_m$; we call these subgroups \textit{skew} subgroups of $\Ree(3) \times C_m$. \\

We start by computing the genus of the quotient curve $\tilde{\mathcal{R}}_q/H$ when $H$ is a non-skew subgroup of $\Ree(3) \times C_m$ that has not been already considered in \cite{GMQZ}. According to Theorem \ref{theorem:classification_subgroups_Ree}, it remains to study the case of $H = \PSL(2, 8) \times C_n$ with $n$ dividing $m$ and the case of $H = K \times C_n$, with $K$ subgroup of $N_2$ and $n$ dividing $m$.

\begin{theorem}\label{theorem:PSL28}
  Let $H = \PSL(2, 8) \times C_n$, with $n$ dividing $m$. Then the genus of the quotient curve $\tilde{R}_q/H$ is
  $$ g_H = \frac{(q^3+1)(q-2) - \Delta_H}{1008 n} + 1, $$
  with
  $$ \Delta_H = 63nq + 56m(q+3q_0+4) + 287n + 216 (\gcd(7, n) - 1)m + (n-1)(q^3+1). $$
\end{theorem}

\begin{proof}
Using GAP, one obtains that $\PSL(2, 8)$ has one element of order $1$, 63 elements of order $2$, 56 elements of order $3$, 216 elements of order $7$ and 168 elements of order $9$. Moreover, its elements of order $3$ are cubes of elements of order 9, hence they lie in the center of a Sylow $3$-subgroup (see the proof of Lemma 3 in \cite{LN}). From Theorem \ref{theorem_48}:
  \begin{align*}
    \Delta_H = & 1 \cdot (n-1)(q^3+1) + 63 \cdot n(q+1) + 56 \cdot (m(q+3q_0+1)+1 + (n-1)) + \\
               & 216 \cdot (\gcd(7, n) - 1)m + 168 \cdot ((m+1) + (n-1)).
  \end{align*}
  The conclusion follows from the Riemann--Hurwitz formula.
\end{proof}

The lattice of conjugacy classes of subgroups of $N_2$ is represented in Figure \ref{fig:lattice_n2_times_cm};
each conjugacy class $\mathcal{H}_i$ in the figure is a class of subgroups of $N_2$ of order $i$.
In particular, $\mathcal{H}_{168}$ and $\mathcal{H}_1$ have size $1$ and contain $N_2$ and $\{ \id \}$ respectively. The figure was computed using GAP.

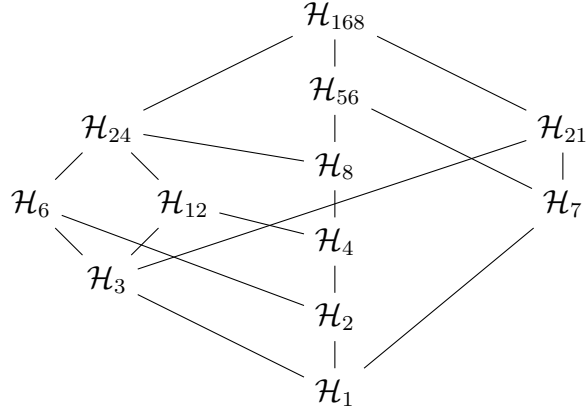
\begin{figure}[h]
  \centering
  \begin{tikzpicture}
    \node at (0, 0) (H1) {$\mathcal{H}_1$};
    \node at (0, 5) (H168) {$\mathcal{H}_{168}$};
    \node at (0, 1) (H2) {$\mathcal{H}_2$};
    \node at (0, 2) (H4) {$\mathcal{H}_4$};
    \node at (0, 3) (H8) {$\mathcal{H}_8$};
    \node at (0, 4) (H56) {$\mathcal{H}_{56}$};
    \node at (-3, 1.5) (H3) {$\mathcal{H}_3$};
    \node at (-4, 2.5) (H6) {$\mathcal{H}_6$};
    \node at (-2, 2.5) (H12) {$\mathcal{H}_{12}$};
    \node at (-3, 3.5) (H24) {$\mathcal{H}_{24}$};
    \node at (3, 3.5) (H21) {$\mathcal{H}_{21}$};
    \node at (3, 2.5) (H7) {$\mathcal{H}_7$};
    \draw[] (H1) -- (H2);
    \draw[] (H2) -- (H4);
    \draw[] (H4) -- (H8);
    \draw[] (H8) -- (H56);
    \draw[] (H56) -- (H168);
    \draw[] (H1) -- (H3);
    \draw[] (H3) -- (H6);
    \draw[] (H3) -- (H12);
    \draw[] (H6) -- (H24);
    \draw[] (H12) -- (H24);
    \draw[] (H24) -- (H168);
    \draw[] (H1) -- (H7);
    \draw[] (H7) -- (H21);
    \draw[] (H21) -- (H168);
    \draw[] (H56) -- (H7);
    \draw[] (H24) -- (H8);
    \draw[] (H21) -- (H3);
    \draw[] (H12) -- (H4);
    \draw[] (H6) -- (H2);
  \end{tikzpicture}
  \caption{Lattice of conjugacy classes of subgroups of $N_2$.}
  \label{fig:lattice_n2_times_cm}
\end{figure}

Subgroups in class $\mathcal{H}_{21}$ or in class $\mathcal{H}_6$ are also subgroups of $\hat{N}_+$.
Thus, only the cases of $K$ a subgroup of $N_2$ in class $\mathcal{H}_{168}$, $\mathcal{H}_{56}$, $\mathcal{H}_{24}$, $\mathcal{H}_{12}$, $\mathcal{H}_8$ or $\mathcal{H}_4$ are left.

\begin{theorem}\label{theorem:N2nonskew}
  Let $H = K \times C_n$, with $K$ a subgroup of $N_2$ of order $168$, $56$, $24$, $12$, $8$ or $4$ and $n$ dividing $m$. Then the genus of the quotient curve $\tilde{R}_q/H$ is
  $$ g_H = \frac{(q^3+1)(q-2) - \Delta_H}{2n|K|} + 1, $$
  with
\begin{equation*}
  \Delta_H=
  \left\{
    \begin{aligned}
       & 7nq + 56m(3q_0+1)  \\
       & \quad + 119n + 48(\gcd(7, n) - 1)m+(n-1)(q^3+1)   && \text{if } |K| = 168, \\
       & 7n(q+1) + 48(\gcd(7, n) - 1)m + (n-1)(q^3+1) && \text{if }  |K| = 56, \\
       & 7nq + 8m(3q_0+1) + 23n + (n-1)(q^3+1) && \text{if }  |K| = 24, \\
       & 3nq + 8m(3q_0+1) + 11n + (n-1)(q^3+1) && \text{if }  |K| = 12, \\
       & 7n(q+1) + (n-1)(q^3+1) && \text{if }  |K| = 8, \\
       & 3n(q+1) + (n-1)(q^3+1) && \text{if }  |K| = 4. \\
    \end{aligned}
    \right.
\end{equation*}

\end{theorem}

\begin{proof}
  Assume $|K| = 168$ first, so that $K = N_2$; using GAP, we computed that $K$ has one element of order $1$, 7 elements of order $2$, 56 elements of order $3$, 56 elements of order $6$ and 48 elements of order $7$. Moreover, its elements of order $3$ do not lie in the center of a Sylow $3$-subgroup. From Theorem \ref{theorem_48}:
  \begin{align*}
    \Delta_H = & 1 \cdot (n-1)(q^3+1) + 7 \cdot n(q+1) + 56 \cdot (m(3q_0+1)+1 + (n-1)) + \\
               & 56 \cdot n + 48 \cdot (\gcd(7,n)-1)m.
  \end{align*}
  The conclusion follows from the Riemann--Hurwitz formula. The argument for the cases of $K$ having order $56$, $24$, $12$, $8$ or $4$ is similar and is omitted.
\end{proof}

This completes the study of non-skew subgroups of $\Ree(3) \times C_m$. We now turn our attention to skew subgroups. If $7$ does not divide $m$, then
$\gcd(|\Ree(3)|,m)=1$, since $\Ree(3)$ has order $1512 = 2^3 \cdot 3^3 \cdot 7$. Hence in this case no skew subgroups exist.
It is easy to see that $7$ divides $m$ if and only if $s \equiv 2 \pmod{6}$ or $s \equiv 3 \pmod{6}$.
%
%
In the remainder of this subsection, we will assume that $m$ is divisible by $7$. We start by narrowing down where skew subgroups of $\Ree(3) \times C_m$ can be located.

\begin{lemma}\label{lem:nonskew1}
Let $n$ be a divisor of $m$. Any maximal subgroup of $\Ree(3) \times C_n$ is non-skew. The same is true for the maximal subgroups of $\PSL(2,8) \times C_n.$
\end{lemma}
\begin{proof}
Let $H \subseteq \Ree(3) \times C_n$ be a maximal subgroup. If $\gcd(n,7)=1$, then Lemma \ref{lem:max} implies immediately that $H$ is non-skew.
If $\gcd(n,7)=7$, the same lemma implies that if $H$ is skew, the projection map $\pi_1: H \to \Ree(3)$ is surjective. Hence for any $\sigma \in \Ree(3)$, there exists $\tau^k \in C_n$ such that $\sigma\tau^k \in H.$
But then for any $\sigma \in \Ree(3)$,
$\sigma^m=(\sigma\tau^k)^m \in H$. In particular, any $2$-Sylow and $3$-Sylow subgroup of $\Ree(3)$ is contained in $H$. Since $\Ree(3)$ does not contain a proper subgroup of cardinality a multiple of $2^3\cdot3^3$,
its maximal groups having orders $504$, $168$, $54$, and $42$, we see that $H$ contains $\Ree(3)$. This implies that $H$ was non-skew after all.

A very similar argument works for maximal subgroups of $\PSL(2,8) \times C_n.$ First one deduces that any skew subgroup contains any Sylow $2$- and $3$-group of $\PSL(2,8)$.
However, the maximal subgroups of $\PSL(2,8)$ have order $2\cdot 7$, $2\cdot 3^2$ (coming from dihedral groups $D_7$, $D_9$) or $2^3\cdot 7$ (coming from a point-stabilizer).
Hence the only subgroup of $\PSL(2,8)$ whose order is a multiple of $2^3\cdot 3^2$, is $\PSL(2,8)$ itself. Hence also in this case, $H$ is non-skew.
\end{proof}

This lemma combined with Lemma \ref{lem:max} and the last part of Theorem \ref{theorem:classification_subgroups_Ree} implies that a skew subgroup of $\Ree(3)\times C_m$ is a subgroup of
$N_2 \times C_m$, $\hat{F} \times C_m$, or of $\hat{N}_+ \times C_m$. Since $\gcd(|\hat{F}|,m)=1$ and skew subgroups of $\hat{N}_+ \times C_m$ are necessarily contained in $C_7 \times C_m$ (shown similarly as
Proposition \ref{prop:subgroups2}), it is enough to investigate skew subgroups of $N_2 \times C_m$.
We start by considering maximal subgroups of this group. The subgroup lattice of $N_2$ depicted in Figure \ref{fig:lattice_n2_times_cm} will be very convenient.
For any $i$, we denote by $H_i$ a representative subgroup from the conjugation class $\mathcal H_i.$
\begin{lemma}\label{lem:skew1}
Let $n$ be a divisor of $m$. Any skew subgroup of $N_2 \times C_n$ has a conjugate contained in $H_{56} \times C_m$.
\end{lemma}
\begin{proof}
If $\gcd(n,7)=1$, no skew subgroups of $N_2 \times C_n$ exist and there is nothing to prove. Assume therefore that $7|n$.
Write $C_n=\langle \tilde{\tau}\rangle$ and
$N_2= \langle s_1, s_2, s_3, l, r\rangle,$ with $s_1, s_2, s_3$ of order $2$ generating a $2$-Sylow subgroup, $l$ of order $3$, $r$ of order $7$, satisfying $l r l^{-1} = r^2$ and several other relations that we will not need.

Then we can choose $H_{56} := \langle s_1, s_2, s_3, r \rangle$, $H_{24} := \langle s_1, s_2, s_3, l \rangle$ and $H_{21} := \langle l,r \rangle$.
Now first suppose that $H$ be a maximal subgroup of $N_2 \times C_n$ that is skew. Then, reasoning as in the proof of Lemma \ref{lem:nonskew1}, we can conclude that after a suitable conjugation $H_{24}$ is contained in $H$
and that for any $\sigma \in N_2$, there exists $k$ such that $\sigma\tilde{\tau}^k \in H$. In particular $r \tilde{\tau}^a \in H$ for some integer $a$.
It follows that $H \ni (l r \tilde{\tau}^a l^{-1} )(r \tilde{\tau}^a)^{-1}= r.$ But then $N_2 \subseteq H$, so $H$ is not skew after all.
We may conclude that any maximal subgroup of $N_2 \times C_n$ is non-skew and therefore up to conjugation is of the form $H_{21} \times C_n$, $H_{24} \times C_n$, $H_{56} \times C_n$, $N_2 \times C_{n/7}$, or $N_2 \times C_{n/p}$,
for some prime number $p$ distinct from $7$, dividing $n$.

Using Lemma \ref{lem:max}, we see that any skew subgroup of $N_2 \times C_n$ has a conjugate contained in $H_{21} \times C_n$, $H_{56} \times C_n$, or $N_2 \times C_{n/p}$. In the latter case, the above argument can be iterated
leading to the conclusion that
a skew subgroup of $N_2 \times C_{n/p}$ has a conjugate contained in $H_{21} \times C_n$ or $H_{56} \times C_n$. Now, with a similar reasoning as before, one can show that up to conjugation all maximal subgroups of
$H_{21} \times C_n$ are $\langle l \rangle \times C_n$, $\langle r \rangle \times C_n$,
$H_{21} \times C_{n/7}$, and $H_{21} \times C_{n/p}$, for some prime number $p$ distinct from $7$, dividing $n$.
Hence any skew subgroup of $H_{21} \times C_n$ has a conjugate contained in $\langle r \rangle \times C_n \subset H_{56} \times C_n$. The lemma now follows.
\end{proof}
%

It remains to analyze the skew subgroups of $H_{56} \times C_m$, where $H_{56} = \langle s_1, s_2, s_3, r \rangle$ is the unique subgroup of $N_2$ in the conjugation class $\mathcal{H}_{56}$.

\begin{theorem}\label{theorem:N2skew}
  Assume $7 \mid m$. All skew subgroups of $H_{56} \times C_m$ are either of the form $H_{i, w} := \langle s_1, s_2, s_3, r \tau^{i w} \rangle$ or of the form $H_{i, w}^\prime := \langle r \tau^{i w} \rangle$,
  for $1 \leq i \leq 6$ and $7w \mid m$. Define $n := \frac{m}{7w}$. The genus of the quotient curve $\tilde{R}_q/H_{i, w}$ is
  $$ g_{H_{i, w}} = \frac{(q^3+1)(q-n-1) - 7n(q+1) - \delta}{16m} \cdot w + 1, $$
  where $\delta = 0$ if $7 \mid n$ and $\delta = 48m$ if $7 \nmid n$. The genus of the quotient curve $\tilde{R}_q/H_{i, w}^\prime$ is
  $$ g_{H_{i, w}^\prime} = \frac{(q^3+1)(q-n-1)  - \delta^\prime}{2m} \cdot w + 1, $$
  where $\delta^\prime = 0$ if $7 \mid n$ and $\delta^\prime = 6m$ if $7 \nmid n$.
\end{theorem}

\begin{proof}
Observe that $H_{56}$ can be presented as $H_{56} := \langle s_1, s_2, s_3, r \rangle$ as in the proof of Lemma \ref{lem:skew1} with $s_1^2=s_2^2=s_3^2=\id=r^7$, $rs_1=s_2r$, $rs_2=s_3r$, and $rs_3=s_2s_3rr$.
The $s_i$ commute with each other, meaning that the group $\langle s_1, s_2, s_3\rangle$ is an elementary abelian $2$-group of order eight.
The group $H_{56}$ has exactly one subgroup of index $7$, namely $\langle s_1, s_2, s_3\rangle$, the Sylow 2-subgroup of $N_2$ of order $8$. By definition
$N_2$ is the normalizer of a $2$-Sylow subgroup of $\Ree(3)$. Since the $2$-Sylow subgroup of $N_2$ is contained in $H_{56}$, it is also normal in $H_{56}$.
Hence $H_{56}$ has one element of order $1$ and seven elements of order $2$. One can check that the remaining $48$ elements have order seven.
This means that any cyclic subgroup of $H_{56}$ of order $7$ is conjugated with $\langle r \rangle$. In particular, such subgroups are not normal. From Figure \ref{fig:lattice_n2_times_cm} we conclude that a nontrivial
normal subgroup of $H_{56}$, necessarily is contained in the $2$-Sylow subgroup $\langle s_1,s_2,s_3\rangle$. In particular $[H_{56},H_{56}] \subset \langle s_1,s_2,s_3\rangle$, where $[H_{56},H_{56}]$ denotes
the commutator subgroup of $H_{56}$. Also, Sylow's theorem implies that the normalizer of $\langle r \rangle$ in $H_{56}$ is $\langle r \rangle$ itself, since it needs to have index eight in $H_{56}$.
As a consequence, the only elements in $H_{56}$ that commute with $r$ are powers of $r$.

Now let $H \subset H_{56} \times C_m$ be a skew subgroup and denote by $\pi_1:H \to H_{56}$ the projection on the first coordinate.
If $s \in \pi_1(H)$ has order two, then $s \in H$, since for any $\tau^a \in C_m$, $s=(s\tau^a)^m \in H$.
Since $H$ is skew, this implies that there exists an element $\tilde{r} \in H_{56}$ of order seven and $\tau^a \in C_m$ such that $\tilde{r}\tau^a \in H$.
Since all subgroups of order seven in $H_{56}$ are conjugated, this means that a conjugate of $H$ contains an element of the form $r^i\tau^a$ for some $i \in \{1,\dots,6\}$.
Redefining $H$ to be that conjugate, we conclude that $r\tau^b \in H$ for some positive integer $b$, by taking a suitable power of the element $r^i\tau^a$.
Now let $n_2$ the smallest positive integer such that $\tau^{n_2} \in H$ and $0 \le a < n_2$ be the smallest nonnegative integer such that $r\tau^a \in H$.

If $H=\langle r\tau^a,\tau^{n_2} \rangle$, then since $H$ is skew, we have $a>0$. Note that from Lemma \ref{lem:triple}, we may conclude that
$n_2|7a$ and since $0<a<n_2$ implies $n_2 \not| a$, we have $n_2/7|a$. In particular, $a=i n_2/7$ for some $i=1,\dots,6$, which implies that a suitable power of $r\tau^a$ is equal to $\tau^{n_2}$.
Here we used $\gcd(i,m)=1$ for $i=1,\dots,6$. We conclude that $H=\langle r\tau^a\rangle$. Defining $w=n_2/7$, we see that $H=H_{i,w}^\prime$.
Note that the genus of $\tilde{\mathcal R}_q/H_{i,w}^\prime$ is the same as the one corresponding to the subgroup of $C_m \times C_m$ with standard exponents $(m/7,n_2,a)$. Therefore, that genus can be obtained directly from
Theorem \ref{theorem:genus_cm_times_cm_Ree}. Writing $p_1=7$, first of all note that if $p_\ell \neq 7$, then $\nu_{d,\ell}=v_{p_\ell}(n_2)$, since $n_1=m/7$ and $n_2/7 | a$. Also note that since $q$ has multiplicative order six modulo
$m$ and we assume that $7$ divides $m$, the element $q$ is a primitive element modulo seven. Moreover, we have $\frac{m}{7}q^d-a=\frac{n_2}{7}\left(\frac{m}{n_2}q^d-i\right)$ for some $1 \le i \le 6$.
Hence, we see that if $7|m/n_2$, then $\nu_{d,1}=v_7(n_2/7)=v_7(n_2)-1$ for all $d$, while if $7 \not| m/n_2$, then there exists exactly one $d$ such that $\nu_{d,1}=v_7(n_2)$, while $\nu_{d,1}=v_7(n_2)-1$ for the remaining values of $d$.
Combining the above, we see that
$$\Delta_{H_{i,w}^\prime} = \left( \frac{m}{7w} - 1 \right) \cdot (q^3 + 1) \quad \text{if $7$ divides $\frac{m}{n_2}$}$$
and
$$\Delta_{H_{i,w}^\prime} = \left( \frac{m}{7w} - 1 \right) \cdot (q^3 + 1) + \left( 7 - 1 \right) \cdot m \quad \text{if $7$ does not divide $\frac{m}{n_2}$}.$$
The stated genus formula for $g_{H_{i,w}^\prime}$ now follows.

Now suppose that $\langle r\tau^a\rangle \subsetneq H$ and choose an element $g\tau^b \in H \setminus \langle r\tau^a\rangle$ for some $g\in H_{56}$.
Then $H \ni r\tau^a g\tau^b \tau^{-a}r^{-1}\tau^{-b}g^{-1}=rgr^{-1}g^{-1}$ is an element of the commutator subgroup of $H_{56}$ and hence in $\langle s_1,s_2,s_3\rangle$. We denote this element by $s$.
If $s=rgr^{-1}g^{-1}=\id$, then $g$ would have been a power of $r$, but then $g\tau^b \in \langle r\tau^a\rangle=\langle r\tau^a,\tau^{n_2}\rangle$ from the definition of $a$ and $n_2$.
Hence $H$ contains the element $s \in\langle s_1,s_2,s_3\rangle\setminus \{\id\}$. Now conjugation by $r\tau^a$ permutes the elements of the set $\langle s_1,s_2,s_3\rangle\setminus \{\id\}$
with a permutation of order a divisor of seven. Since $rs_1r^{-1}=s_2$, the action is not trivial. Hence conjugation by $r\tau^a$ cyclically permutes the elements of the set $\langle s_1,s_2,s_3\rangle\setminus \{\id\}$.
In particular, we may conclude that $\langle s_1,s_2,s_3\rangle \subset H,$ implying that $H=\langle s_1,s_2,s_3,r\tau^a\rangle=H_{i,w}.$ The genus computation is now very similar as before. The eight cyclic subgroups
$\langle sr\tau^a\rangle$, with $s \in\langle s_1,s_2,s_3\rangle\setminus \{\id\}$, give rise to the contribution $\delta$ to $\Delta_{H_{i,w}}$, while the elements of the form $s \tau^{jn_2}$ give rise to the contribution $7m/n_2(q+1)$.
The remaining elements of the form $\tau^{jn_2}$ give the contribution $(m/n_2-1)(q^3+1)$.
  \end{proof}

The following genera, obtained using Theorems \ref{theorem:PSL28} and \ref{theorem:N2nonskew} for $s = 1$, are new up to our knowledge. Note that for $s=1$ Theorem \ref{theorem:N2skew} cannot give new genera, because $7$ does not divide
$m$ in this case.

\begin{table}[h]
  \centering
  \begin{tabularx}{\textwidth}{llX}
    \hline
    s & Field & Genus \\
    \hline
    1 & $\mathbb{F}_{3^{18}}$ & 445, 4393 \\
    \hline
  \end{tabularx}
  \caption{New genera for $s=1$ from Theorems \ref{theorem:PSL28} and \ref{theorem:N2nonskew}.}
  \label{tab:Ree3timesCmskew}
\end{table}

\section*{Acknowledgements}

The first and third author would like to acknowledge the support from The Danish Council for Independent Research (DFF-FNU) for the project \emph{Correcting on a Curve}, Grant No.~8021-00030B.


\end{document}